\documentclass[14pt]{article}

\usepackage{amsmath}

\numberwithin{equation}{section}
\usepackage{amsfonts}
\usepackage{amsthm}
\usepackage{amssymb}
\usepackage{bbm}
\usepackage{centernot}
\usepackage[x11names]{xcolor}
\usepackage{chemfig}
\usepackage{amsmath}
\usepackage{cite}
\usepackage[nodayofweek]{datetime} 
\usepackage{dsfont}
\usepackage{enumitem}
\usepackage{euscript}
\usepackage{faktor}
\usepackage{float}
\usepackage{geometry}
\usepackage{graphicx}
\usepackage{mathrsfs}
\usepackage{mathtools}
\usepackage{polynom}
\usepackage{stmaryrd}
\usepackage{tikz}
\usepackage{tikz-cd}
\usepackage{wasysym}
\usepackage{xfrac}
\usepackage[x11names]{xcolor}
\usepackage{mathtools}
\usepackage{setspace}
\usepackage[all, cmtip]{xy}
\usepackage{comment}
\usepackage{pgfplots}
\pgfplotsset{compat=1.15}
\usepackage{mathrsfs}
\usetikzlibrary{arrows}
\usepackage{thmtools}

\usepackage{hyperref}

\newlist{legal}{enumerate}{10}
\setlist[legal]{label*=\arabic*.}

\newif\ifproofread

\DeclarePairedDelimiter\abs{\lvert}{\rvert}
\DeclarePairedDelimiter\norm{\lVert}{\rVert}

\makeatletter
\newcommand*\bigcdot{\mathpalette\bigcdot@{.5}}
\newcommand*\bigcdot@[2]{\mathbin{\vcenter{\hbox{\scalebox{#2}{$\m@th#1\bullet$}}}}}
\makeatother

\usepackage{mathtools}
\DeclarePairedDelimiter\ceil{\lceil}{\rceil}

\newtheorem*{Lemma*}{Lemma}
\newtheorem*{Corollary*}{corollary}
\newtheorem*{theorem*}{Theorem}
\newtheorem*{definition*}{Definition}

\newtheorem{theorem}{Theorem}[section]

\newtheorem{definition}[theorem]{Definition}
\newtheorem{construction}[theorem]{Construction}
\newtheorem{lemma}[theorem]{Lemma}

\newtheorem{corollary}[theorem]{Corollary}
\newtheorem{claim}[theorem]{Claim}
\newtheorem{conj}[theorem]{Conjecture} 
\newtheorem{question}[theorem]{Open Question} 

\declaretheoremstyle[
  headfont=\normalfont\bfseries,%\itshape,
  numbered=unless unique,
  bodyfont=\normalfont,
  spaceabove=1em plus 0.75em minus 0.25em,
  spacebelow=1em plus 0.75em minus 0.25em,
]{hartending}

\declaretheorem[
  style=hartending,
  title=Remark,
  sibling=theorem,
]{remark}

\newcommand{\thistheoremname}{}
\newtheorem*{genericthm}{\thistheoremname}

\newcommand{\EE}{\mathbb{E}}

\newcommand{\NN}{\mathbb{N}}

\newcommand{\PP}{\mathbb{P}}
\newcommand{\QQ}{\mathbb{Q}}

\newcommand{\RR}{\mathbb{R}}

\newcommand{\ZZ}{\mathbb{Z}}

\newcommand{\cK}{\mathcal{K}}

\newcommand{\cO}{\mathcal{O}}

\newcommand{\stab}{\operatorname{stab}}
\newcommand{\diag}{\operatorname{diag}}
\newcommand{\bd}{{\rm d}}

\newcommand{\on}[1]{\operatorname{#1}}

\newcommand{\cov}[1]{\operatorname{cov}(#1)}

\newcommand{\bra}{\left\langle}
\newcommand{\ket}{\right\rangle}

\newcommand{\einv}{\mathcal M(X_n)^A_e}
\newcommand{\cinv}{\mathcal M(X_n)^A_c}

\newcommand{\pid}{\mathrel{\ooalign{$\lneq$\cr\raise.22ex\hbox{$\lhd$}\cr}}}

\AtBeginDocument{}

\DeclareMathOperator{\supp}{supp}

\title{Tori Approximation of Families of Diagonally Invariant Measures}
\date{}
\author{
    Omri Nisan Solan \footnote{\texttt{omrinisan.solan@mail.huji.ac.il}}
    \and 
    Yuval Yifrach\footnote{\texttt{yuval@campus.technion.ac.il}}
}
\Huge
\begin{document}
\maketitle

%	We approximate any portion of any orbit of the full diagonal group $A$ in the space of unimodular lattices in $\mathbb{R}^n$ using 
%    a fixed proportion of a compact $A$-orbit.
%    Using those approximations for the appropriate sequence of orbits, 

\begin{abstract}
   Let $A$ be the full diagonal group in $\on{SL}_{n}(\RR)$. We study possible limits of Haar measures on periodic $A$-orbits in the space of unimodular lattices $X_n$. 

  We prove the existence of %\yinote{I removed part of the abstract.}
  non-ergodic measures which are also weak limits of these compactly supported $A$-invariant measures. 
  In fact, given any countably many $A$-invariant ergodic measures, we show that there exists a sequence of Haar measures on periodic $A$-orbits such that the ergodic decomposition of its weak limit has these measures as factors with positive weight. 
   In particular, we prove that any compactly supported $A$-invariant and ergodic measure is the weak limit of the restriction of different compactly supported periodic measures to a fixed proportion of the time.
   
%    The compactly supported $A$-invariant measures are supported on compact $A$-orbits, corresponding to a sequence of number fields $K_i$, which are independent of the measures to appear in the ergodic decomposition. This implies that a compactly supported $A$-invariant measure may appear as an ergodic component in the limit of other compactly supported $A$-invariant measures, with unrelated number fields.

  In addition, for any $c\in (0,1]$ we find a sequence of Haar measures on periodic $A$ orbits that converges weakly to $cm_{X_n}$ where $m_{X_n}$ denotes the Haar measure on $X_n$. In particular, we prove the existence of partial escape of mass for Haar measures on periodic $A$ orbits. These results give affirmative answers to questions posed by Shapira in ~\cite{ShapiraEscape}. Our proofs are based on a modification of Shapira's proof in ~\cite{ShapiraEscape} and on a generalization of a construction of Cassels, as well as on effective equidistribution estimates of Hecke neighbors by Clozel, Oh and Ullmo in \cite{HeeOhHeckeOld}.
\end{abstract}
\section{Introduction}
Let $(X_n, d_{X_n})$ denote the metric space of unimodular lattices in $\RR^n$ and let $A < \on{SL}_n(\RR)$ denote the subgroup of diagonal matrices with positive diagonal inside the group of real matrices with determinant 1. Among the ergodic $A$-invariant probability measures on $X_n$, those which are supported on compact orbits stand out due to their connection to number theory. 
Indeed, there is a surjective map between the set of full $\ZZ$-modules inside totally real number fields and the set of compact $A$-orbits in $X_n$. For any totally real degree $n$ number field $K$, let $\sigma_i: K\hookrightarrow\RR;i=1,\dots,n$ be some ordering of the natural embeddings of $K$ and let $\sigma=(\sigma_1,\dots,\sigma_n): K\rightarrow\RR^n$ denote their concatenation. The compact orbit corresponding to a full module $M\subset K$ is then the $A$-orbit of the unimodular lattice $\operatorname{cov}(\sigma(M))^{-1/n}\sigma(M)$ where $\on{cov}(\Lambda)$ denotes the co-volume of a lattice $\Lambda$. This orbit is indeed compact, and every compact $A$-orbit is given in this manner.
% Indeed, there is a surjective map between the set of full modules inside totally real number fields and the set of compact $A$-orbits in $X_n$. \\
Denote:

\begin{equation*}
    \mathcal M(X_n)=\{\text{finite measures on $X_n$}\};	
\end{equation*}
\begin{equation*}
    \mathcal M(X_n)^A_e=\{\text{ergodic }A\text{-invariant probability measures on $X_n$}\};
\end{equation*}
\begin{equation*}
    \mathcal M(X_n)^A_c=\{\text{ergodic }A\text{-invariant probability measures on $X_n$ supported on compact orbits}\}.
\end{equation*}
We endow $\mathcal M(X_n)$ with the weak-$\ast$ topology and study the closure $\overline{\cinv}$ in it. 
It is known that apart from measures of compact orbits, $m_{X_n}, 0\in \mathcal M(X_n)$. The Haar measure is obtained by applying ergodicity to the result of Benoist and Oh \cite{benoist2007equidistribution}, or more directly using Shapira and Zheng's \cite{shapira2021translates}. The trivial measure is obtained by Shapira \cite{ShapiraEscape}. 
In this paper, we show that $\overline{\cinv}$ contains several families of natural and important measures, thus answering positively the discussion in \cite[\S 1.5]{ELMV1} and negatively Conjecture \cite[Conjecture 1.11]{ELMV1}. 

\subsection{Measures with Predetermined Ergodic Factors}
The first family of measures we consider contains non-ergodic measures. Given any finitely many elements $\nu_1,\dots,\nu_k\in \einv$, we find a measure in $\mu\in \overline{\cinv}$ such that the ergodic decomposition of $\mu$ has positive weight on each of the $\nu_i$'s for $i=1,\dots,k$. In ~\cite[p.4 Q1, Q2]{ShapiraEscape}, Shapira asks whether non-ergodic measures and measures with a given periodic $A$-invariant ergodic factor can be obtained as weak limits of $\cinv$ elements. The following Theorem shows in particular that this is indeed the case:
\begin{theorem}\label{thm: ergodicDecomp}
    Let $(\nu_i)_{i=1}^N \subseteq \einv$ be a possibly infinite sequence of $A$-invariant ergodic measures. Then there exists $\mu\in \overline{\cinv}$ such that the ergodic decomposition of $\mu$ has positive weights on the $\nu_i$'s. In particular for any $\mu\in \cinv$ there exists $\tilde{\mu}\in \overline{\cinv\setminus\{\mu\}}$ such that the ergodic decomposition of $\tilde \mu$, contains $\mu$ with positive weight. 
\end{theorem}
\begin{remark}[The weights in the ergodic decomposition]\label{rem: uniformRN}
    For a single measure $\nu_1$, we get an explicit lower bound on the weight, namely $\mu \ge n\left(\frac{\lfloor n/2\rfloor}{4n^2(n^2-1)}\right)^{n-1}\nu_1$. 
    Generally, the proof fixes a sequence of positive real numbers $(c_i)_{i=1}^N$ with $\sum_{i=1}^N c_i = 1$ and constructs $\mu$ with $\mu \ge \left(\frac{\lfloor n/2\rfloor}{4n^2(n-1)^2(n+1)}\right)^{n-1} c_i^{n-1}\nu_i$. 
\end{remark}
% \begin{theorem}
%     To compute the Theta, note that \[\Theta / n\left(\frac{\lfloor n/2\rfloor}{4n^2(n^2-1)}\right)^{n-1} = \frac{vol(S_{\frac{1}{n-1}w_0})}{vol(\{x\in \RR^{n-1}_0: |x_i-x_j|<1\})} = \frac{\frac{1}{\sqrt n}\frac{1}{(n-1)^{n-1}}}{\sqrt n}\]
%     Hence $\Theta = \left(\frac{\lfloor n/2\rfloor}{4n^2(n-1)^2(n+1)}\right)^{n-1}
% \end{theorem}
% \begin{remark}[Countably many measures]
%     Using the same technics one can prove the same result for countably many measures. For possibly infinite sequence $\nu_1,\nu_2,\dots$ of measures and every sequence of positive constants $c_1+c_2+\cdots = 1$ we can 
%     we can construct $\mu$ with $\mu \ge\Theta(c^{n-1}_i)\nu_i$.
% \end{remark}
\subsection{Partial Escape of Mass and Entire Mass Approximations}
The next measures we will find in $\overline{\cinv}$ are measures with total mass strictly between $0$ and $1$. In his paper ~\cite[p.4, Q2]{ShapiraEscape}, Shapira proves that the zero measure is contained in $\overline{\cinv}$ and poses the question of whether measures with total mass strictly between $0$ and $1$ also lie there. 
In \cite{david2018dirichlet} David and Shapira find a sequence of orbits with a lower bound on the escape of mass but did not prove that the limit is not zero. 
We answer Shapira's question in the affirmative in two ways. The first is the following corollary of Theorem \ref{thm: ergodicDecomp}:
\begin{corollary}\label{cor: tailoredEscape}
	For every sequence $\nu_1,\nu_2, \dots, \nu_k,\ldots\in \einv$ there exists $\mu\in \overline{\cinv}$ such that $\mu(X_n)\in (0,1)$ and the ergodic decomposition of $\mu$ has positive weights on $\nu_i$ for all $i$.
\end{corollary}
The next family of measures we find inside the closure of the compactly supported periodic measures is the family $\{cm_{X_n}: c\in (0,1]\}$ where $m_{X_n}$ is the Haar probability measure on $X_n$. This family of measures is substantially different from the measure families we approximated so far since we can describe their entire ergodic decomposition (which is simply $cm_{X_n}$) rather than describe some of the factors and have no control over the remaining factors. 
In \cite{shapira2021translates}, they show that $m_{X_n}\in \overline{\cinv}$. 
Here we prove a more general theorem.
\begin{theorem}\label{thm: haarApproximation}
	Let $c\in (0,1]$. Then $cm_{X_n}\in \overline{\cinv}$.
\end{theorem}
In particular, the above theorem gives another answer to Shapira's question regarding the partial escape of mass. 
\begin{remark}[Volume of the compact orbits]\label{rem: size regulator}
Every compact orbit $Ax$ has a volume $\on{Vol}(Ax)=\on{Vol}(\RR^{n-1}/\log\stab_A(x))$ and a corresponding order $R_x$. The compact orbits we construct for the proof of Theorem \ref{thm: ergodicDecomp} have volume $O(\log \on{Disc}(R_x))$. As written, our proof of Theorem \ref{thm: haarApproximation} doesn't provide estimates on the volumes. 
\begin{comment}
Heuristically, the number field \href{https://en.wikipedia.org/wiki/Artin%27s_conjecture_on_primitive_roots}{Artin's Conjecture},
implies $\on{reg}(R) = \on{disc}(R)^{c'}$ where $c'>0$ depends on $c$ for almost any $R$. 	
\end{comment}
  In Remark \ref{rem: one Hecke suffices and polynomial dependence in discriminant}, we argue that similarly to the proof of Theorem \ref{thm: haarApproximation}, one can choose the approximating compact orbits $Ax$ to have volume $O(\on{Disc}(R_x)^{c'})$ for some $c'>0$.
\end{remark}
The estimate of the regulator in Remark \ref{rem: size regulator} that provides a sequence of orbits with volume polynomial in the discriminant of the corresponding order and partial escape of mass, contradicts \cite[Conjecture 1.11]{ELMV1}, which states that every such sequence of orbits must converge to an algebraic measure. The spirit of the conjecture is not disproved, and one could hope that the conjecture can be corrected by considering only partial escape of mass. We do not believe this is the case for general orbits.
Instead, we provide the following conjecture:
\begin{conj}\label{conj: big orbit behavior}
    Let $Ax_i\subset X_n$ be a sequence of compact orbits corresponding to orders $R_i$.  Suppose that $\on{Vol}(Ax_i) > \on{disc}(Ax_i)^\varepsilon$ for some $\varepsilon>0$. Here $\on{disc}(Ax_i)$ is the discriminant of the orbit $Ax_i$. It is defined as in \cite{ELMV1}. 
    Then $\lim_{i\to \infty} \mu_{Ax_i}$ is a combination of countably many periodic algebraic orbits of groups $H_j\supsetneq A$, with total mass bounded from below as a function of $\varepsilon$. 
    Moreover, if $K_i = R_i\otimes \QQ$, and $(\on{disc}(R_i) / \on{disc} K_i) < \on{disc}(R_i)^{\delta}$ for some $\delta=\delta(\varepsilon)$, then $\lim_{i\to \infty} \mu_{Ax_i}$ is algebraic. 
    We suppose the same will hold for general arithmetic $G/\Gamma$ with the appropriate generalization of $\on{disc}(R_i) / \on{disc} K_i$. The reader may compare this discussion to a similar discussion carried out in \cite[Corollary 1.7]{ELMV1}.
\end{conj}

\subsection{Method of Proof}\label{ssec: method of proof}
We will first show a sketch of the proof of Theorem \ref{thm: ergodicDecomp}, 
for the particular case $k=1$ where we approximate a single measure $\mu$.
We will first show how to do that depending on a conjecture. 
Then we will show how to overcome the need to use the conjecture.

\begin{conj}\label{conj: number fields packets equidistributed}
    Let $K$ be a number field. For every $I \le \mathcal{O}_K$, we get a measure $\mu_I$ on a compact orbit associated to $I$ (See Definition \ref{def: inv measures}). 
    The measure $\mu_I$ depends only on the class of $I$ in the class group ${\rm Cl}_{\mathcal{O}_K}$. 
    It is believed that 
    \begin{align}\label{eq: packets}
      \mu_K := \frac{1}{\#{\rm Cl}_{\mathcal{O}_K}} \sum_{I\in {\rm Cl}_{\mathcal{O}_K}} \mu_I \xrightarrow{\on{disc}(K)\to \infty} m_{X_n},
    \end{align}
    where $m_{X_n}$ is the Haar measure on $X_n$, and $\on{disc}(K)$ is the discriminant of $K$.
    Moreover, we expect that $d_{\mathcal{M}^1(X_n)}(\mu_K, m_{X_n}) = O(\on{disc}(K)^{-\star})$, for $\star>0$, where $d_{\mathcal{M}^1(X_n)}$ is the natural metric on the space of probability measures $\mathcal{M}^1(X_n)$. 
\end{conj}
In the following paragraph we will use the $\Theta$ notation. In this notation, we write $f=\Theta(g)$ for two functions $f,g$, if there exist constants $C_1,C_2>0$ such that $f<C_1g$ and $g<C_2f$. 

Another ingredient of the proof is a construction of a number field with very small units compared to its discriminant.
A construction of Shapira \cite{ShapiraEscape} following Cassels \cite{cassels1952product} is a series of number fields $K$ whose unit groups are generated by elements $u\in \mathcal{O}_K^\times$ such that $\log |u| = \Theta(\log {\rm disc}(K))$ (which is exponentially small compared to the bound given by the class number formula: $\log |u| = O({\rm disc}(K)^{1/(n-1)})$, provided that the unit lattice is sufficiently balanced).
Applying Conjecture \ref{conj: number fields packets equidistributed} to these number fields, we get that the different measures $\mu_I$ for $I\in {\rm Cl}_{\mathcal{O}_K}$ are equidistributed. Using this equidistribution we can approximate a generic point $x\in X_n$ of our desired ergodic measure $\mu$, namely there is $I\in {\rm Cl}_{\mathcal{O}_K}$ such that $d_{X_n}(x, x') = O({\rm disc}^{-\star}(K))$ for some $x'\in \supp \mu_I$. 
Hence for $r=\Theta(\log({\rm disc}(K)))$ ball in $A$, $a\in B_A(r)$ we have that $ax$ is close to $ax'\in \supp(\mu_I)$. 
On the other hand, the orbit $ax'$ is not much larger. The periods are controlled by the size of the units, which are again $\Theta(\log({\rm disc}(K)))$. This would imply that we can approximate every ball in an $A$-orbit with measures $\mu_I$ for $I$ in the class group of the special number field $K$. 

To overcome the need to use Conjecture \ref{conj: number fields packets equidistributed}, we prove a weaker form on a special collection of orders. We take a measure $\mu_I$ as above and apply to it a $p$-Hecke operator $T_p$. On the one hand, Hecke operators are known to have well-behaved equidistribution properties (See Clozel, Oh and Ullmo \cite{HeeOhHeckeOld} for an effective equidistribution result on Hecke operators and \cite{benoist2007equidistribution,Shapira-Aka} for other results using equidistribution of Hecke operations applied to closed orbits). We expect to have $f(T_p\mu_I, m_{X_n}) = O(p^{-\star})$. 
On the other hand, the resulting measure $T_p\mu_I$ is the average of measures on compact orbits, corresponding to modules of a sub-order $R=\ZZ+p\mathcal{O}_K \subset \mathcal{O}_K$ of index $[\mathcal{O}_K: R]=p^{n-1}$. 

While this guarantees that the orbits would be equidistributed, to bound the sizes of the orbits we tweak the construction of $K$. 
Not only $\mathcal{O}_K$ should have units of logarithmic size, but also $R$. This guarantees that $T_p\mu_I$ is an average of measures of logarithmically small compact orbits, and enables the proof to work without Conjecture \ref{conj: number fields packets equidistributed}.

We will now describe the construction of Shapira \cite[Section 5.1]{ShapiraEscape}, following Cassels \cite{cassels1952product}, and our modification.
\begin{construction}
    For any tuple $\overline m=(m_1,\dots,m_n)\in \ZZ^n$ define the polynomial $R_{\overline m}=\prod_{i=1}^n(x-m_i)-1$. Under mild conditions on $\overline m$ we get that $R_{\overline m}$ is irreducible.
    In the number field $\QQ[x]/(R_{\overline m})$, the elements $x-m_i$ are units. 
    % Moreover, this construction gives control over their inverses (they are precisely of the form $\prod_{i\neq i_0}(x-m_i)$ for some $i_0$). 
    The description of these units guarantees that they have polynomial size compared to the discriminant of the number field.    
    % The explicit way $R_{\overline m}$ is defined allows Shapira to control the sizes (of the embeddings of) the units $x-m_i$ and of the discriminant of $R_{\overline m}$. These estimates show that the size of the units is very small compared to the discriminant of $R_{\overline m}$. 
\end{construction}
Shapira uses the structure of the units to show that the compact $A$-orbits corresponding to $\ZZ[x]/(R_{\overline m})$ spend a high percentage of their time near the cusp (see Theorem \ref{thm: Shapira escape}).

In our construction, we use a similar construction of polynomials: 
\begin{construction}
    Let $p$ be a prime number.
    For any tuple $\overline m=(m_1,\dots,m_n)\in \ZZ^n$ define the polynomial $R_{\overline m, p}=\frac{\prod_{i=1}^n(px-m_i)-1}{p^n}$. Under some modular conditions on $\overline m$ we get that $R_{\overline m, p}$ is integral and irreducible.
    In the number field $K=\QQ[x]/(R_{\overline m, p})$, the elements $px-m_i$ are units. 
    The description of these units guarantees that they have polynomial size compared to the discriminant of the number field, but also lies in $\ZZ + p\mathcal O_K$.
    % The explicit way $p_{\overline m}$ is defined allows Shapira to control the sizes (of the embeddings of) the units $x-m_i$ and of the discriminant of $p_{\overline m}$. These estimates show that the size of the units is very small compared to the discriminant of $p_{\overline m}$. 
\end{construction}
This allows us to find a small compact orbit $Ax$ such that, after applying the Hecke operator to it, the new collection of compact orbits is effectively equidistributed in $X_n$ and is composed of many small compact $A$ orbits.
\begin{remark}
    Conjecture \ref{conj: number fields packets equidistributed} seems very complicated. 
    It was proven in case $n=2$ in \cite{duke1988hyperbolic}. 
    A non-effective version of it is proven for $n = 3$ under some splitting restriction in \cite{einsiedler2011distribution}. The proof uses a deep and nontrivial number theoretic result, namely, the sub-convexity estimate of $L$ functions of cubic fields. 
\end{remark}

The proof of Theorem \ref{thm: haarApproximation} goes along similar lines. We start with a compact orbit $Ax_0$, constructed using the techniques used in \cite{ShapiraEscape}. As in \cite{ShapiraEscape}, this orbit will lie in the cusp for a density $1$ proportion of its lifetime. To achieve a measure with partial escape of mass, we apply two Hecke operators. The first Hecke operator is responsible for pulling an appropriate proportion of the orbit out of the cusp. The second operator is used to get equidistribution from the section of the orbit outside of the cusp (this can be done with one operator, see Remark \ref{rem: one Hecke suffices and polynomial dependence in discriminant}).
However, this is not a single $A$-orbit. Instead, we prove by ergodicity that at least one of the $A$-orbits that constitute the Hecke operators applied to the orbit has to be almost equidistributed as well. 
\begin{comment}
As in \cite{ShapiraEscape}, this compact orbit will lie in the cusp for most of its lifetime. However, the points of the orbits distributes continuously in the cusp, in the sense that some are further in the cusp then others. 
Then we take some Hecke operator $T$. It pulls points out of the cusp in fixed rate, that is, if $\lambda_1(x)$ was very small then 
on average over $x'\in T(x_0)$ we have $\lambda_1(x') = \Theta(\min(p^\alpha \lambda_1(x_0), 1))$. 
We choose $T$ such that $c$ proportion of $T(Ax)$ will leave the cusp while the remaining $1-c$ proportion of $T(Ax)$ will stay in the cusp. 
We will use the Hecke equidistribution to show that the part that leaves the cusp equidistributes. 
A possible problem is that $T(Ax_0)$ is composed of possibly many $A$-orbits. 
To solve that we arrange the following additional property: 
a positive proportion of $x_0$-s $T$-Hecke neighbors will be on the same $A$ orbit. This implies that $T(Ax_0)$ contains a long $A$-orbit, $Ax_1$. 
Now we use ergodicity to deduce that the distribution estimate of $T(Ax_0)$ holds for its long $A$-orbit subset $Ax_1$.
This orbits will give us Theorem \ref{thm: haarApproximation}. 
We further remark that the method of finding a Hecke neighbor whose diagonal orbit occupies positive proportion of the entire Hecke neighbors is also used in another paper by Uri Shapira and Menny Aka, \cite{Shapira-Aka}.
\end{comment}

\begin{remark}[Comparison to \cite{benoist2007equidistribution} and \cite{shapira2021translates}]\label{rem: compariason to hee and shap}
    Plugging $c=1$ to Theorem \ref{thm: haarApproximation} yields a construction identical to the one in \cite{benoist2007equidistribution} (with some irrelevant extra steps). We begin with the Hecke operator of a compact orbit and get a new collection of $A$-orbits. We obtain its equidistribution and then use ergodicity to sample one of them that equidistributes.
    %  applying Hecke operator to a compact orbit, we obtain a collections of compact orbits which   
    % yields examples similar in nature to those of \cite{shapira2021translates}. Both are obtained from Hecke neighbor of a Minkovski embedding of a number field. However our analysis is simpler in this case, as a consequence of restricting the number field we use.
    The construction in \cite{shapira2021translates} is of the same kind, however, they show a way to sample a specific compact orbit that equidistributes, not using ergodicity in the same fashion.
\end{remark}

\subsection{Further Research}
The first natural improvement of Theorem \ref{thm: ergodicDecomp} is the full approximation. 
\begin{question}
    Let $\mu \in \cinv$. Can it be the limit measure of other measures in $\cinv$? In other words, is $\mu \in \overline{\cinv\setminus \{\mu\}}$?
\end{question}
% This question is more complicated than improving the constants in this paper. 
It can be seen that simply improving the bounds we give in this paper cannot give a positive answer to this question, for the following reason.
If $n=3$, $Ay$ is a compact orbit and $Ax$ is an orbit. We can consider the set $B = \{a\in A: d_{X_3}(ax, Ay)<\delta\}$. 
The connected components of $B$ are roughly hexagons. It can be seen that two such hexagons $H_1, H_2$ cannot have $R$ long parts of the boundaries which are $\delta R$ close to one another, for some $\delta>0$ and all $R>0$ sufficiently large.

Focusing on a different aspect of this work, we would like to find even less trivial limits of large $A$-orbits in the following sense:
\begin{question}
    What are the possible limits $\lim_{i\to \infty} \mu_{Ax_i}$ where $Ax_i$ is a compact orbit corresponding to an order $R_i$ satisfying that $\on{reg}(R_i) > \on{disc}(R_i)^{\varepsilon}$?
    Can it be a non-ergodic $A$-invariant probability measure?
    Can it contain as an ergodic component a compact orbit uniform measure $\mu_{Ay}$?
\end{question}

Another possible direction for improvement is by working in different homogenous spaces. 
Many questions can be asked, and here we choose to focus on a single homogenous space:
\begin{question}
  Let $B$ be an $n^2$-dimensional central simple algebra over $\QQ$ which splits over $\RR$. Let $\cO_B$ be an order in $B$. Let $\Gamma = \on{SL}_1(\cO_B) \subseteq \on{SL}_n(\RR)$ a lattice. 
  There are many compact orbits in $\on{SL}_n(\RR)/\Gamma$, coming from degree-$n$ totally real number fields $K\subset B$. 
  What are the possible limits of these compact orbits? Can Haar measure be obtained? 
  Can you extend Theorem \ref{thm: ergodicDecomp} to this setting? Can you compute an explicit non-Haar limit measure (in contrast to Theorem \ref{thm: ergodicDecomp} where we have control only over a small portion of the measure)?
\end{question}

\subsection{Acknowledgments}
The second author would like to express his deep gratitude to Uri Shapira for his support and encouragement. Without them, this paper would not exist. We thank Uri Shapira for bringing the main questions answered in this paper to our attention and for many intriguing discussions with him which contributed a lot to this paper. Moreover, the first author thanks Andreas Wieser and Elon Lindenstrauss for many fruitful discussions. The first author thanks the generous donation of Dr. Arthur A. Kaselemas.
The second author acknowledges the support of ISF grants number 871/17. 
This work has received funding from the European Research Council (ERC) under the European Union’s Horizon 2020 research and innovation programme (Grant Agreement No. 754475), and ERC grant HomDyn, No. 833423.
This work is part of the first author's Ph.D. thesis.
\section{Notation and Preliminaries}
\begin{definition}[$O$-notations]\label{def: O notations}
    For two real functions $f,g$ on a set $A$ we write $f\ll g$  if there exists a constant $C$ independent on the parameters of $f$ and $g$ such that $|f|\leq Cg$ on $A$. 
    The notation $O(g)$ will refer to some implicit function $f$ which satisfies $f\ll g$. 
    The notation $\Theta(g)$ will refer to some implicit function $f$ which satisfies $g \ll f\ll g$. 
    Whenever $r$ is a parameter going to $0$ or $\infty$, the notation $o_r(g)$ will refer to some implicit function $f$ which satisfies $f\ll g\cdot h$, for some implicit function $h \to 0$ as $r$ goes to $0$ or $\infty$ respectively. 
\end{definition}
Fix $\norm{\cdot}$ to denote the $\ell^2$ norm on $\RR^n$. 
Given a lattice $\Lambda\subset \RR^n$ we use $\cov{\Lambda}$ to denote the co-volume of $\Lambda$.
Let $X_n$ denote the space of unimodular lattices in $\RR^n$ and let $d_{X_n}(\cdot,\cdot)$ %\yinote{this notation is used also for the measures...}
denote the Riemannian metric on $X_n = \on{SL}_n(\RR)/\on{SL}_n(\ZZ)$ coming from the right invariant Riemannian metric $d_{\on{SL}_n(\RR)}(\cdot, \cdot)$. 
Let $\RR^{n-1}_0=\{v\in \RR^n:\sum_iv_i=0\}$. We abuse notations and define $\exp=\exp\circ \diag:\RR^{n-1}_0\to A$ to be the standard parametrization.
We denote by $m_{X_n}$ probability measure on $X_n = \on{SL}_n(\RR)/\on{SL}_n(\ZZ)$ coming from the Haar measure on $\on{SL}_n(\RR)$.
\begin{definition}[Space of Measures]\label{def: space of measures}
	Let $\mathcal M(X_n)$ denote the space of finite measures on $X_n$ endowed with the topology induced by $\mu_k\rightarrow \mu$ if for any $f\in C_c(X_n)$ it holds that $\mu_k(f)\rightarrow \mu(f)$. 
    % We define a metric on $\mathcal M(X_n)$ which induces this topology by letting, for any $\mu_1,\mu_2\in \mathcal M(X_n)$:
	% \begin{equation}
	% 	d(\mu_1,\mu_2)=\sup_{\eps>0} \varepsilon \sup\left\{\abs{\int f\bd\mu_1-\int f \bd\mu_2}:f\in C_c(X_n)\text{ is $1$-Lipschitz and is supported on }\cK_{\eps}\right\}
	% \end{equation}
	% where $\mathcal K_{\eps}=\{x\in X_n: \lambda_1(x) \ge \eps\}$.
    % \osnote{This is not a metric on the right space. This is a Frechet space but not a banach one. 
    % We should mention that here.}
    % \yinote{Are you sure that the weak * topology on probability measures is metrizable? In any case, we are not really using this metric, except for Conjecture \ref{conj: number fields packets equidistributed}. What do you say we give it up?}
    % \osnote{The weak * toplogy on prob measures is metrizable. The weak * topology on finite measures is also metrizable, but without a homogeneous metric. }
\end{definition}
The following definition is relevant to \S\ref{sec: gluing lemma}. Denote by $L, U\subset \on{SL}_n(\RR)$ the subgroups of lower and upper triangular matrices with diagonal $1$. 

\begin{definition}[Special subgroups of $\on{SL}_n(\RR)$]
    \label{def: subgroups}
    Let 
    \begin{equation}
        w_0 = \left(\frac{2i-n-1}{2}\right)_{i=1}^n\in \RR^{n-1}_0, \qquad a_t = \exp(tw_0) \text{ for all }t\in \RR,
    \end{equation}
    and note that 
    \begin{align*}
        L = \{g\in \on{SL}_n(\RR): a_{-t}ga_{t}\xrightarrow{t\to \infty}I\},\quad
        U = \{g\in \on{SL}_n(\RR): a_tga_{-t}\xrightarrow{t\to \infty}I\},
    \end{align*}
    that is, $L,U$ are the expanding and contracting horospheres with respect to $a_t$.         
\end{definition}

We now present formally the relation between compact $A$-orbits in $X_n$ and subgroups of number fields. 
\begin{definition}\label{def: compact orbit}
    For every degree $n$, totally real number field $K$, denote by ${\rm Lat}_K'$ the set of free $\ZZ$-modules of rank $n$ in $K$. We define an equivalence relation on ${\rm Lat}_K'$ by identifying two lattices $\Lambda_1,\Lambda_2 \subset K$ if $\Lambda_1 = k\Lambda_2$ for some $k\in K^\times$. The quotient space is denoted by ${\rm Lat}_K$, and for every $\Lambda\in {\rm Lat}_K'$, denote by $[\Lambda]\in {\rm Lat}_K$ its equivalence class.
    For every rank $n$, $\ZZ$-module $\Lambda\in {\rm Lat}_K'$ consider the lattice $x_\Lambda := \sigma(\Lambda)/(\cov{\sigma(\Lambda)})^{1/n} \in X_n$, where $\sigma_i: K\hookrightarrow\RR;i=1,\dots,n$ is some ordering of the natural embeddings of $K$ and let $\sigma=(\sigma_1,\dots,\sigma_n): K\rightarrow\RR^n$ denote their concatenation. 
    Denote by $\cO_\Lambda = \{k\in K: k\Lambda\subseteq \Lambda\}$. This is a ring. Denote by $\cO_\Lambda^{\times, >0} = \{u\in \cO_\Lambda^\times:\sigma_i(u)>0: i=1,\dots,n\}.$
    For every $U\subseteq \cO_K^{\times, >0}$ denote $A_U = \{\diag(\sigma_1(u), \sigma_2(u),\dots,\sigma_n(u)): u\in U\}$. 
    Note that these definitions depend implicitly on the ordering of the real embeddings of $K$.
\end{definition}
\begin{theorem}\label{thm: all compact orbits}
    Fix $n\ge 0$.
    For every totally real number field $K$ of degree $n$, $[\Lambda] \in {\rm Lat}_K$, and choice of the ordering of the real embeddings of $K$,
    the orbit $Ax_\Lambda$ is compact, independent of the representative $\Lambda\in {\rm Lat}_K'$ of $[\Lambda]\in {\rm Lat}_K$ and this is a one to one parametrization of all compact $A$-orbits in $X_n$. 
\end{theorem}
This theorem is equivalent to \cite{mcmullen2005minkowski}, apart from the part of this parametrization being one-to-one, which is Folklore and we do not use it in this paper.
\begin{definition}\label{def: inv measures}
    For every point $x\in X_n$ such that the orbit $Ax$ is compact, denote $\mu_{Ax}$ the $A$-invariant measure probability on $Ax$. 
    For every $[\Lambda]\in {\rm Lat}_K$ denote $\mu_\Lambda = \mu_{Ax_\Lambda}$. 
\end{definition}
\section{Effective Approximations}\label{2}
In this section, we prove an effective approximation theorem of points in $X_n$ by points of compact orbits. Formally, the goal is to prove the following lemma.
\begin{lemma}\label{lem: fastApproximation2}
    For every compact set $\cK\subset X_n$ there exists $C_{\cK}>0$ and a sequence $t_k\xrightarrow{k\to \infty}\infty$ such that the following holds:
    for any $y\in \cK$ there exists a sequence $(y_k)_k\subset X_n$ such that $d_{X_n}(y_k, y) < C_{\cK}\exp\left(-\frac{\lfloor n/2\rfloor}{4n(n^2-1)}t_k\right)$ and such that $y_k$ is stabilized by a subgroup $\exp(\Lambda_k)\subseteq A$ where $\Lambda_k\subseteq \RR^{n-1}_0$ is generated by vectors $(v^{(i)}_k)_{i=1}^{n-1}$ such that $\left|\left( v_k^{(i)} \right)_j - (1 - \delta_{ij}n)t_k\right| = O(1)$.
\end{lemma}
\begin{remark}\label{rem: approximation of parts of orbits}
    Lemma \ref{lem: fastApproximation2} implies that $Ay_k$ is compact with total volume $n^{n-3/2}t^{n-1}_k+O(t_k^{n-2})$ in the $\RR^{n-1}_0$ parametrization. In addition, it implies that for every $v\in \RR^{n-1}_0$ and $y,y_k$ as in the lemma, we have:
    \begin{align}\label{eq: when close}
        d_{X_n}((\exp v).y_k,(\exp v).y) \le O\left(\exp\left(\max_{1\le i,j\le n}|v_{i}-v_j|-\frac{\lfloor n/2\rfloor}{4n(n^2-1)}t_k\right)\right).
    \end{align}
    Since the set $B_0 = \{v\in \RR^{n-1}_0: \max_{1\le i,j\le n}|v_{i}-v_j|\le 1\}$ has volume $\on{vol}(B_0) = \sqrt n$, we deduce that a portion of $n\left(\frac{1}{n}\frac{\lfloor n/2\rfloor}{4n(n^2-1)}\right)^{n-1} + O(1/t_k)$ of the orbit $Ay_k$ approximates the the orbit $Ay$. 
\end{remark}

The proof of Lemma \ref{lem: fastApproximation2} will be composed of two components. The first is an approximation of lattices via Hecke neighbors. We will define Hecke neighbors in Definition \ref{def: Hecke} and obtain a good approximation by them in Lemma \ref{lem: heckeApproximation}.
The second component generates infinitely many points in $X_n$ with compact $A$-orbits, with a control on the geometry of their stabilizers in $A$ and the stabilizers of their sublattices of a given index. This gives us the arsenal of points on which we apply Lemma \ref{lem: heckeApproximation} to deduce Lemma \ref{lem: fastApproximation2}.
% \begin{lemma}\label{compactOrbitsConstruction}
%     There exists $C=C(n)>0$ and a compact subset $K\subset X_n$ such that for any prime $p\equiv 1\mod n$ large enough as a function of $d$, there exists $x_p\in K\cap \Omega_n$, $R_p>0$ and $T_p\in \on{SL}_{n-1}(\RR)$ such that:
%     \begin{legal}
%         \item \label{18:22} The set	
%         \begin{equation}
%             \Delta_p=R_p\cdot(T_0^{-1}T_pT_0)\cdot\Delta_*
%         \end{equation}
%         is a fundamental domain for the $A$-action on every sublattice of $x_p$ of index $p$ and $\norm{T_p}\leq C$;
%         \item \label{2:54} For every $v\in \Delta_p$, $\norm{v}\leq p^C$.
%     \end{legal}
% \end{lemma}

\subsection{Hecke Density}
The source of our good approximation comes from the quantitative version of the equidistribution of Hecke neighbors. 
There are many references to the equidistribution of Hecke neighbors such as ~\cite[Theorem 1.1]{HeeOhHeckeOld}, ~\cite[Theorem 3.7]{HeeOhHeckeNew},~\cite[Theorem 1.2]{NonQuantitativeHeckeEquidistribution}.
In this section, we will cite the result of Laurent Clozel, Hee Oh and Emmanuel Ullmo \cite{HeeOhHeckeOld} and deduce the approximation result we need.
\begin{definition}[Function space]\label{def: L^2_0}
    Let $L^2_0(X_n) = \{f\in L^2(X_n, m_{X_n}): \int_{X_n} f \bd m_{X_n} = 0\}$ be the Hilbert space of $L^2$ functions on $X_n$ with $0$ mean.
\end{definition}
\begin{definition}[Definition of the $p$-Hecke neighbors and the Hecke operator]\label{def: Hecke}
    For every sequence of nonnegative integers $k_1\le k_2\le \dots\le k_n$ consider:
    \[a = a_{p; k_1, k_2, \dots, k_n} = \frac{1}{p^{(k_1+\dots+k_n)/n}}\diag(p^{k_1}, p^{k_2}, \dots, p^{k_n})\in \on{SL}_n(\RR).\]
    For every $x=g\on{SL}_n(\ZZ)\in X_n$ denote $T_a(x) = g\on{SL}_n(\ZZ)a \on{SL}_n(\ZZ)$. This set is finite since $a\on{SL}_n(\ZZ)a^{-1}$ is commensurable to $\on{SL}_n(\ZZ)$.
    The size $\#T_a(x) = \#(\on{SL}_n(\ZZ)a\on{SL}_n(\ZZ)/\on{SL}_n(\ZZ))$ depends only on $k_1,\dots,k_n$ and not on $x$.
    Equivalently, 
    $$T_a(x) = \left\{\frac{1}{\sqrt[n]{\cov{x'}}} x': x'\subseteq x \text{ with }x/x' \cong \ZZ/p^{k_1}\ZZ \oplus\cdots \oplus \ZZ/p^{k_n}\ZZ\right\}.$$
    For every function $f\in L^2(X_n)$ define the Hecke action on functions \[T^{\rm F}_a(f)(x) = \frac{1}{\#T_a(x)} \sum_{x'\in T_a(x)}f(x')\in L_0^2(X_n).\]
    For every point $x\in X_n$, define the Hecke action on points
    \[T_a^{\rm M}(x) := \frac{1}{\#T_a(x)}\sum_{x'\in T_a(x)}\delta_{x'}\]
    and for every measure $\mu$ on $X_n$ define the Hecke action on measures
    \[T_a^{\rm M}(\mu) := \int_{X_n}T_a^{\rm M}(x)\bd\mu(x).\]
\end{definition} 
% To prove Lemma \ref{lem: heckeApproximation}, we will have to estimate the following quantity:
% \begin{definition}
% 	For any $x\in X_n$ and $k\in\NN$ we define:
% 	\begin{equation}
% 		r_{k}(x)=\inf{\{\varepsilon>0:\text{ for any }y\in X_n\text{ }N_k(x)\cap B_{\varepsilon}^{X_n}(y)\neq \emptyset\}}.
% 	\end{equation}
% \end{definition}
% Indeed, one can see that Lemma \ref{lem: heckeApproximation} follows from an estimate of the following form (which still needs to be proved):
% \begin{equation}\label{heckeEstimate}
% 	r_{k}(x)\leq C(x)k^{-\alpha}.
% \end{equation}
% For a given $x\in X_n$ and sequence $(k_m)_m\subset \NN$ going to infinity, the limiting distribution as $m\rightarrow \infty$ of the sets $N_{k_m}(x)$ insider $X_n$ has been studied in the literature. In fact, it was studied in a more general context which we do not require here (see ~\cite[Theorem 1.1]{HeeOhHeckeOld},~\cite[Theorem 3.7]{HeeOhHeckeNew},~\cite[Theorem 1.2]{NonQuantitativeHeckeEquidistribution}). The following theorem follows from any of these references:
The following theorem is a particular case of \cite[Theorem 1.1]{HeeOhHeckeOld}, specialized for $\on{SL}_n$ as in \cite[Example 5.1]{HeeOhHeckeOld}.
\begin{theorem}\label{thm: equidistributionOfHecke}
    For every prime $p$ and $k_1\le k_2\le \dots\le k_n, a = a_{p; k_1, k_2, \dots, k_n}\in \on{SL}_n(\RR)$ as in Definition \ref{def: Hecke}, the operator norm of $\left.T^{\rm F}_a\right|_{L^2_0(X_n)}$ is bounded by:
    % gamma_i(a) = p^{k_{n+1-i}-k_i}
    \[\left\|\left.T^{\rm F}_a\right|_{L^2_0(X_n)}\right\| \le \prod_{i\le n/2} \frac{1}{p^{(k_{n+1-i}-k_i)/2}}\frac{(k_{n+1-i}-k_i)(p-1) + (p+1)}{p+1} \le  p^{-\frac{1}{2}\sum_{i=1}^{n/2}(k_{n+1-i}-k_i)} \cdot C(k_1,\dots,k_n),\]
    where $C(k_1,\dots,k_n)$ depends polynomially on $k_1,\dots,k_n$.
\end{theorem}
\begin{lemma}\label{lem: heckeApproximation}
    For every compact subset $\cK\subset X_n$ there exists $C = C(\cK) > 0$ such that for any $x, y\in \cK$ and $p, k_1\le k_2\le \dots\le k_n, a = a_{p; k_1, k_2, \dots, k_n}\in \on{SL}_n(\RR)$ as in Definition \ref{def: Hecke}, there exists a Hecke neighbor $x'\in T_a(x)$ such that:
    \begin{equation}
            d_{X_n}\left(x',y\right)\leq C(\cK)\left\|\left.T^{\rm F}_a\right|_{L^2_0(X_n)}\right\|^{1/(n^2-1)}.
    \end{equation}
\end{lemma}
\begin{proof}
    Recall the right invariant Riemannian metric $d_{\on{SL}_n(\RR)}$ on $\on{SL}_n(\RR)$, and its descend to $X_n$, the metric $d_{X_n}$. 
    Let $r_0<\min(\on{inj}(x),\on{inj}(y))$, where $\on{inj}(x)$ is the injectivity radius of $x$, that is, the maximal radius $r$ such that the translation map $g\mapsto gx$ is injective on $B_{\on{SL}_n(\RR)}(I; r)$, itself being the radius $r$ ball around the identity in $\on{SL}_n(\RR)$.
    Let $f_x=\chi_{B_{X_n}(x; r_0)}$ be the indicator of a radius $r_0$ ball around $x$.
    Then $\int_{X_n}f_x\bd m_{X_n} = \on{vol}(B_{\on{SL}_n(\RR)}(I; r_0)) = \Theta(r_0^{n^2-1})$, and a similar equality holds for $f_y=\chi_{B_{X_n}(y; r_0)}$.
    Denote by $v_{r_0} = \on{vol}(B_{\on{SL}_n(\RR)}(I; r_0))$. 
    Then $\tilde f_x = f_x - v_{r_0}\in L^2_0(X_n)$ and 
    $\tilde f_y = f_y - v_{r_0}\in L^2_0(X_n)$
    have norm $\|f_x\|^2 = \|f_y\|^2 = v_{r_0}(1-v_{r_0})^2 + (1-v_{r_0})v_{r_0}^2 = (1-v_{r_0})v_{r_0}$. 
    Consider $\bra \tilde f_x, T^{\rm F}_a(\tilde f_y)\ket$. 
    On the one hand, it is at most $\left\|\left.T^{\rm F}_a\right|_{L^2_0(X_n)}\right\| \|\tilde f_x\|\|\tilde f_y\| = \left\|\left.T^{\rm F}_a\right|_{L^2_0(X_n)}\right\|(1-v_{r_0})v_{r_0}$. 
    On the other hand, consider the set: 
    \[T_a^{-1}(B_{X_n}(y; r_0)) := \{x'\in X_n : T_a(x')\cap B_{X_n}(y; r_0)\neq \emptyset\}\]
    and note that $T^{\rm F}_a(\tilde f_y)|_{(T_a^{-1}(B_{X_n}(y; r_0)))^c} \equiv -v_{r_0}$. 
    Assume that $T_a^{-1}(B_{X_n}(y; r_0))\cap B_{X_n}(x; r_0) = \emptyset$, and get that $T^{\rm F}_a(\tilde f_y)|_{B_{X_n}(x; r_0)} \equiv -v_{r_0}$.
    It follows that:
    \begin{align*}
        \bra \tilde f_x, T^{\rm F}_a(\tilde f_y)\ket &= \int_{X_n}\tilde f_xT^{\rm F}_a(\tilde f_y)\bd m_{X_n} = \int_{X_n}f_xT^{\rm F}_a(\tilde f_y)\bd m_{X_n} = 
        \int_{B_{X_n}(x; r_0)}T^{\rm F}_a(\tilde f_y)\bd m_{X_n}\\& = v_{r_0}\cdot (-v_{r_0}) = -v_{r_0}^2.
    \end{align*}
    We deduce that \[v_{r_0}^2\le \left\|\left.T^{\rm F}_a\right|_{L^2_0(X_n)}\right\| \cdot \|\tilde f_x\|\cdot \| \tilde f_y\| = \left\|\left.T^{\rm F}_a\right|_{L^2_0(X_n)}\right\|(1-v_{r_0})v_{r_0},\] and hence $v_{r_0}\le \left\|\left.T^{\rm F}_a\right|_{L^2_0(X_n)}\right\|$. 
    Using this logic in reverse, we deduce that if $v_{r_0}>\left\|\left.T^{\rm F}_a\right|_{L^2_0(X_n)}\right\|$ then $T_a^{-1}(B_{X_n}(y; r_0))\cap B_{X_n}(x; r_0)\neq \emptyset$, that is, there exist 
    $x' = g_0x, y'=g_1y$ such that $y' \in T_a(x')$ and $g_0, g_1 \in B_{\on{SL}_n(\RR)}(I; r_0)$.
    Thus $g_0^{-1}y' = g_0^{-1}g_1y \in g_0^{-1}T_a(x') = T_a(x)$. On the other hand, 
    $g_0^{-1}y' = g_0^{-1} g_1y$ satisfies 
    $d_{X_n}(g_0^{-1} g_1y, y)\le 2r_0$. 
    
    Altogether, we have proved that if $\left\|\left.T^{\rm F}_a\right|_{L^2_0(X_n)}\right\| < v_{r_0}$ and $r_0 \le \min(\on{inj}(x), \on{inj}(y))$ then there exists $x'' = g_0^{-1} g_1y$ with $x''\in T_a(x)$ and $d_{X_n}(x'', y)\le 2r_0$. 
    % This argument works provided that $v_{r_0} > \left\|\left.T^{\rm F}_a\right|_{L^2_0(X_n)}\right\|$ and $r_0\le \min(\on{inj}(x), \on{inj}(y))$. 
    Now, let $\cK\subset X_n$ be a compact set. Denote the minimum of the injectivity radius on $\cK$ by $r_\cK>0$. 
    If $v_{r_\cK}>\left\|\left.T^{\rm F}_a\right|_{L^2_0(X_n)}\right\|$, then we can find $r_0 = \Theta\left(\left\|\left.T^{\rm F}_a\right|_{L^2_0(X_n)}\right\|^{1/(n^2-1)}\right)$ with $r_0<r_\cK$ and $v_{r_0}>\left\|\left.T^{\rm F}_a\right|_{L^2_0(X_n)}\right\|$, and the desired follows.
    If $v_{r_\cK}\le \left\|\left.T^{\rm F}_a\right|_{L^2_0(X_n)}\right\|$, then the desired follows for $C(\cK) = \frac{\on{diam}(\cK)}{(v_{r_\cK})^{1/(n^2-1)}}$.
\end{proof}

\subsection{Construction of special number field}\label{COA}
\begin{theorem}\label{thm: special field}
  For every prime number $p \equiv 1 \mod 2n$ sufficiently big as a function of $n$, one can find a totally real number field $K$ with the following properties:
  \begin{enumerate}[label=\emph{(\emph{\alph*})}, ref=(\emph\alph*)]
	\item 
		The unit group $\mathcal{O}_K^\times$ contains units $u_1,\dots,u_{n}$ with $u_1 u_2\cdots u_n = 1$. 
	\item One can order the real embeddings $\sigma_1,\dots,\sigma_n:K\to \RR$ such that $\sigma_i(u_j) > 0$ for all $1\le i,j\le n$, and 
		\begin{align}\label{eq: size of logs}
			\log \sigma_i(u_j) = \begin{cases}
				-2n(n-1)\log p + O(1),&  \text{if }i=j,\\
				2n\log p + O(1),&  \text{if }i\neq j.
			\end{cases}
		\end{align}
	\item \label{point: unit in ring} The units $u_i$ lie in the ring $\ZZ+p\mathcal{O}_K$. 
  \end{enumerate}
  
  % one can find a subgroup $\Gamma_p\subseteq \RR^{n-1}_0$ 
  % with $\lambda_{1}(\Gamma_p)\asymp \lambda_{n-1}(\Gamma_p)\asymp \log p$, 
  % and $\Lambda_p\in X_n$ such that for every $x\in \text{supp}(H_p(\Lambda_p))$ we have $\exp(\operatorname{diag}(\Gamma_p)) x = x$. 
\end{theorem}
\begin{proof}
  Since $p\equiv 1 \mod 2n$, the polynomial $x^n + 1$ has a $n$ different solutions mod $p$. By Hensel's Lemma (See \cite{conrad2015hensel}), it has $n$ solutions mod $p^n$, call them $a_0',\dots,a_{n-1}'\in \{1,\dots,p^n - 1\}$.
  Let $a_i = a_i' + 2ip^n$ for $i=0,\dots,n-1$, and note that $p^n +1 \le a_{i+1}-a_i\le 3p^n$ for all $0\le i\le n-2$. 
  consider now the polynomial $R(x) = \frac{(px - a_1)(px - a_2)\cdots (px - a_n) - 1}{p^n}$. It has integer coefficients since by assumption, $\sum_{i_1<\dots<i_k}a_{i_1}\cdots a_{i_k} = 0 \mod{p^n}$ for all $k=1,\dots,n-1$ and $a_1a_2\cdots a_n = (-1)^n \mod{p^n}$. 
  
  To control the real embeddings of the resulting number field, we will approximate the real roots of $R$. 

  By Taylor's theorem (see \cite[\S 20.3]{kline1998calculus}) applied at $a_{i_0}/p$ we have 
  \[R(x) = R(a_{i_0}/p) + R'(a_{i_0}/p) (x-a_{i_0}/p) + \frac{1}{2}R''(x')(x-a_{i_0}/p)^2,\]
  for some $x'\in [a_{i_0}/p, x]$ (this interval notation does not assume $a_{i_0}/p\le x$). 
  Note that 
  \[R(a_{i_0}/p) = - \frac{1}{p^n}, \qquad R'(a_{i_0}/p) = \prod_{j\neq i_0}(a_{i_0}/p-a_j/p) = (-1)^{n - i_0} \alpha_{i_0} p^{n^2-2n+1},\] for some $\alpha_{i_0}\ge 1$, with $\alpha_{i_0}=\Theta(1)$,
  and for all $x'\in [-p^{n-1}, 2np^{n-1}]$ we have 
  \[|R''(x')| = \left|\sum_{1\le k<l\le n}\prod_{j\neq k,l}(x-a_j/p)\right| = O(p^{n^2-3n+2}).\]
  That implies that for all $x\in [-p^{n-1}, 2np^{n-1}]$ we have 
  \begin{align}\label{eq: taylor}
    R(x) = -\frac{1}{p^n} + (-1)^{n-i_0} \alpha_{i_0} p^{n^2-2n+1} (x-a_{i_0}/p) + O(p^{n^2-3n+2})(x-a_{i_0}/p)^2,
  \end{align}
  Choose $x_{i_0}'$ with $x_{i_0}'-a_{i_0}/p = 2\frac{(-1)^{n-i_0}}{p^{n^2-n+1}\alpha_{i_0}}$.
  Note that $|x_{i_0}'-a_{i_0}/p|<1$, and hence $x_{i_0}'\in [-p^{n-1}, 2np^{n-1}]$,
  Then 
  \[
    R(x_{i_0}') = -\frac{1}{p^n} + (-1)^{n-i_0} \alpha_{i_0} p^{n^2-2n+1} (x_{i_0}'-a_{i_0}/p) + O(p^{-n^2-n}) = \frac{1}{p^n} + O(p^{-n^2-n}).
    \]
  If $p$ is sufficiently large as a function of $n$ then $R(x_{i_0}')>0$. 
  Since $R(a_{i_0}) = -\frac{1}{p^n} < 0$, the mean value theorem implies that there is $x_{i_0}\in [a_{i_0}/p,x_{i_0}']$ with $R(x_{i_0}) = 0$. 
  By Eq. \eqref{eq: taylor} we deduce that:
  \begin{align}
    0 = R(x_{i_0}) = -\frac{1}{p^n} + (-1)^{n-i_0} \alpha_{i_0} p^{n^2-2n+1} (x_{i_0}-a_{i_0}/p) + O(p^{n^2-3n+2})(x_{i_0}-a_{i_0}/p)^2.
  \end{align}
  This implies:
  \begin{align*}
    \frac{1}{p^n} &= (x_{i_0}-a_{i_0}/p)\left((-1)^{n-i_0} \alpha_{i_0} p^{n^2-2n+1} + O(p^{n^2-3n+2})(x_{i_0}-a_{i_0}/p)\right) \\&
    = (x_{i_0}-a_{i_0}/p)(-1)^{n-i_0} \alpha_{i_0} p^{n^2-2n+1}\left(1 + O\left(p^{-n^2}\right)\right),
  \end{align*}
  and hence, 
  \[x_{i_0} = a_{i_0}/p + \frac{(-1)^{n-i_0}}{p^{n^2-n+1} \alpha_{i_0}} + O\left(\frac1{p^{2n^2-n+1}}\right).\]
  Since $|a_{i_0}/p-x_{i_0}|<1$ we deduce that all $x_i$ are different, and hence these are the distinct roots of $R$. 
  Since $|a_i-a_j|\ge p^n$ for all $i\neq j$ it follows that $|a_i/p-x_j|\ge p^{n-1} - 1 > 1$. 
  It follows that $R = \prod_{i=1}^n(x-x_i)$ is irreducible, as if $R_0 = \prod_{i\in I}(x-x_i)$ is an integer polynomial for some nonempty $I\subsetneq \{1,\dots,n\}$, then in the ring $\ZZ[x]/(R_0)$, the element $px-a_{i_0}$ is a unit for every $i_0\notin I$, but its norm in $\ZZ[x]/(R_0)$ is $\prod_{i\in I}(px_i-a_{i_0})$, which is greater than $1$ in absolute value.

  The number field $K=\QQ[\alpha]$ where $R(\alpha)=0$ has units $u_i' = p\alpha - a_i$. Let $u_i = (u_i')^2$. 
  The real embeddings of $K$ are $\sigma_i:\alpha\mapsto x_i$. They satisfy $\log \sigma_i(u_j) = 2 \log |px_i-a_j|$. 
  Now Eq. \eqref{eq: size of logs} follows from the properties of $x_i$.
\end{proof}

% \begin{Corollary}
%   For every $\Lambda\in X_n, \varepsilon>0$ there exists $\Lambda_\varepsilon\in X_n$ with 
%   $d(\Lambda,\Lambda_\varepsilon) = O_\Lambda(\varepsilon)$, and $\stab_A(\Lambda)$ has full rank in $A$ and is satisfies 
%   \[
%     \lambda_1(\Gamma_\varepsilon)
%     \asymp \lambda_{n-1}(\Gamma_\varepsilon) \asymp -\log \varepsilon,
%     \]
%     where $\Gamma_\varepsilon = (\exp\circ \diag)^{-1}(\stab_A(\Lambda_\varepsilon))$. 
%   In particular, there exists $\delta$ which depends only on $\Lambda$ for which $d(\mu', (x\mapsto \exp(\diag(x))\Lambda)_*m_{B(-\delta\log \varepsilon)}) = O(\varepsilon^\delta)$, where $\mu'$ satisfies $\delta \mu' + (1-\delta) \mu'' = \mu_{Ax_i}$. 
% \end{Corollary}
% The lattice $\Lambda_p$ and $\Gamma_p$ is using an adaption of Shapira's construction.
% \begin{claim}
%   For every prime number $p \equiv 1 \mod d$, one can find a number field $\mathcal{O}_K$ with independent units $u_1,\dots,u_{n-1}\in \mathcal O_K^\times$ with $\|u_i\| = \Theta(p^?)$ and the projection of
%   $u_1$ to $\mathcal{O}_K/(p)$ lies in $\FF_p$. 
% \end{claim}

% \input{subsection_construction_of_number_field.tex}

\subsection{Application of Hecke operator to compact orbits}
In this section, we introduce the following lemma, which describes how to verify that a Hecke operator splits a compact orbit into sufficiently many compact diagonal orbits.
\begin{lemma}\label{lem: Hecke effect on compact orbit}
    Let $K$ be a totally real number field of degree $n$, let $x_{\cO_K}\in X_n$ the point with compact orbit corresponding to the $\ZZ$-module $\cO_K$ and let $p, k_1\le \dots\le k_n, a = a_{p; k_1, k_2, \dots, k_n}\in \on{SL}_n(\RR)$ as in Definition \ref{def: Hecke}. Let $U\subseteq \mathcal{O}_K^{\times, >0}$ be a subgroup and $A_U\subset A$ the corresponding diagonal subgroup as in Definition \ref{def: compact orbit}. If $U \subseteq \ZZ+p^k \mathcal{O}_K$ and $0\le k_1 \le k_n\le k$ then $A_U \subseteq \stab_A(x')$ for every $x'\in T_a(x_{\cO_K})$. 
\end{lemma}
\begin{proof}
    Let $u\in U$. Since $u\in \ZZ+p^k \mathcal{O}_K$, there is $m\in\ZZ$ such that $u\equiv m\mod{p^k}$. Thus for every $\bar b\in \mathcal{O}_K/(p^k)$ we have $u\bar b = m\bar b$, that is, the multiplication by $u$ action on $\cO_K/(p^k)$ is in fact a multiplication by a scalar.
    Hence, the element $a_u = \diag(\sigma_1(u), \sigma_2(u), \dots, \sigma_n(u))\in A_U$ preserves $x_{\cO_K}$, $p^kx_{\cO_K}$ and acts on the quotient by multiplication by the scalar $m$. This implies that $a_u$ acts on the set of mid-groups $\{\Lambda: p^kx_{\cO_K}\subseteq \Lambda \subseteq x_{\cO_K}\}$ trivially.  Indeed, there is an isomorphism \[\{\Lambda: p^kx_{\cO_K}\subseteq \Lambda \subseteq x_{\cO_K}\}\cong \{\bar\Lambda \subseteq x_{\cO_K}/p^k x_{\cO_K}\},\]
    and $u$ acts on the right-hand side as a multiplication by a scalar, which preserves all groups.
    This implies that $a_u$ preserves all Hecke neighbors, by the second construction (See Definition \ref{def: Hecke}).
\end{proof}
\begin{remark}
    Lemma \ref{lem: Hecke effect on compact orbit} remains true if we replace $x_{\cO_K}$ by $x_\Lambda$ for some $\Lambda\in {\rm Lat}_K$ and $U$ by a subgroup of $\cO_\Lambda$ satisfying $U \subseteq \ZZ+p^k \mathcal{O}_\Lambda$. The proof is similar.
\end{remark}
\begin{corollary}\label{cor: compactOrbitsConstruction2}
    Let $p, k_1\le \dots\le k_n, a = a_{p; k_1, k_2, \dots, k_n}\in \on{SL}_n(\RR)$ as in Definition \ref{def: Hecke}. There is a compact subset $\cK\subset X_n$ such that for any prime $p\equiv 1\mod{2n}$ large enough as a function of $n$, there exists $x_p\in \cK$, and a lattice $\Lambda \subset \RR^{n-1}_0$ such that,
    \begin{enumerate}[label=\textit{\emph{\arabic*.}}, ref=\arabic*]
        \item \label{part: group stabilizing} For every Hecke operator $T_a$ with $0\le k_1\le \dots\le k_n\le 1$ and for every element $y\in T_a(x_p)$ we have that $\exp(\Lambda)\subseteq A$ stabilize $y$. 
        \item \label{part: gens of lattice} The lattice $\Lambda$ is generated by vectors $(v^{(i)})_{i=1}^{n-1}$ satisfying $\left|\left( v^{(i)} \right)_j - (2n - 2n^2\delta_{ij})\log p\right| = O(1)$.
    \end{enumerate}
\end{corollary}
\begin{proof}
    Fix $p\equiv 1\mod {2n}$ sufficiently big for Theorem \ref{thm: special field}. 
    Let $K$ be the number field constructed in Theorem \ref{thm: special field}, and $u_1,\dots,u_n$ the corresponding units. Let $(\sigma_i)_{i=1}^n$ be the real embeddings of $K$.
    Let $\Lambda$ be the subgroup of $\RR^{n-1}_0$ generated by $(\log \sigma_i(u_j))_{i=1}^n$ for $j=1,\dots,n-1$. The bounds on $u_j$ in Theorem \ref{thm: special field} imply Part \ref{part: gens of lattice}. 
    Consider the point $x_{\cO_K}\in X_n$. 
    It follows from \cite[Proposition A.1]{tomanov2003closed} that there is a compact set $\cK\subseteq X_n$ which intersects every $A$-orbit. 
    Hence there is $x_p = mx_{\cO_K}\in \cK$ for some $m\in A$. 
    Since the $\on{SL}_n(\RR)$ action on $X_n$ commutes with $T_a$, to prove part \ref{part: group stabilizing} it is sufficient to prove that $\exp(\Lambda)\subseteq A$ stabilize $y$ for every $y\in T_a(x_{\cO_K})$. 
    This follows from Lemma \ref{lem: Hecke effect on compact orbit} and Property \ref{point: unit in ring} of Theorem \ref{thm: special field} for the units $u_1,\dots,u_k$. 
\end{proof}
% \begin{remark}
%     The use of 
%     We used McMullen \cite[Theorem 4.1]{mcmullen2005minkowski} to find a compact set $\mathcal{K}\subset X_n$ which intersects every compact $A$-orbit, for the simple description of the set $K$. In \cite[Proposition A.1]{tomanov2003closed} Margulis provides a simpler proof, showing that every $A$-orbit intersects another compact set $\mathcal{K}\subset X_n$, less explicit. 
% \end{remark}
\subsection{Proof of Lemma \ref{lem: fastApproximation2}}
% \begin{proof}[Proof Lemma \ref{lem: fastApproximation2}]
    The points we construct are indexed by primes $p\equiv 1\mod{2n}$. Fix such a prime $p$. Denote $t_p = 2n\log p$ and let $x_p$ and $\Lambda$ be as in Corollary \ref{cor: compactOrbitsConstruction2}. 
    Note that Corollary \ref{cor: compactOrbitsConstruction2} implies the exact bounds on a set of generators of $\Lambda$ we need.
    Let $k_1=\dots=k_{\lfloor n/2\rfloor}=0, k_{{\lfloor n/2\rfloor}+1}=\dots=k_n=1$, and $a = a_{p; k_1, k_2, \dots, k_n}\in A$ as in Definition \ref{def: Hecke}. 
    We claim that for some $y_p\in T_ax_p$ we have $d_{X_n}(y,y_p)\ll\exp(-\frac{\lfloor n/2\rfloor}{4n(n^2-1)}t_p)$. 
    Indeed, by Theorem \ref{thm: equidistributionOfHecke} we have $\left\|\left.T^{\rm F}_a\right|_{L^2_0(X_n)}\right\| \le p^{-\lfloor n/2\rfloor/2}$. Thus by Lemma \ref{lem: heckeApproximation}, there is $y_p\in T_a(x_p)$ with \[d_{X_n}(y_p, y) = O(p^{-\frac{\lfloor n/2\rfloor}{2(n^2-1)}}) = O\left( \exp\left( t_p\cdot \frac{\lfloor n/2\rfloor}{4n(n^2-1)} \right) \right). \]
    Note that the implicit constant was provided by Theorem \ref{thm: equidistributionOfHecke} and depends only on the compact set containing $y$.\qed
\section{Approximation of Invariant Measures}\label{sec: gluing lemma}
In this section, we prove Theorem \ref{thm: ergodicDecomp}. We approximate measures by approximating their generic points. When we want to approximate a finite collection of measures simultaneously, we need to find regions of $A$-orbits that approximate all measures on subregions.
% points which are generic for all the measures we approximate together in some sense. 
The following lemma is our tool for this purpose:
\begin{lemma}[Higer-rank closing lemma]\label{lem: diagonal closing lemma}
    For every compact set $\cK\subseteq X_n$ there are compact sets $\cK_A\subset A, \cK_L \subset L, \cK_U\subset U$, such that for every $x_0,x_1\in \cK$ there are $k_A\in \cK_A, k_L\in \cK_L, k_U\in \cK_U$ such that $k_Lk_Ak_Ux_0 = x_1$. 
    Here $A,L, U$ are defined as in Definition \ref{def: subgroups}. 
\end{lemma}
\begin{proof}
    First, we will show that for every $x_0, x_1\in X_n$ we have $x_1\in LAUx_0$. 
    Since $L$ is the expanding horosphere of $a_t$, almost every point in $Lx_1$ is generic with respect to the forward $a_t$ action and the Haar measure $m_{X_n}$. Choose a $a_t$-generic point point $l_1 x_1\in Lx_1$. 
    The $LU$-decomposition implies that the product map $L\times A\times U\to \on{SL}_n(\RR)$ is one-to-one. Since $L\times A\times U$ and $\on{SL}_n(\RR)$ are $n^2-1$ dimensional, this implies that $LAU\subseteq \on{SL}_n(\RR)$ is open.
    Hence $LAUx_0$ is an open set in $X_n$, and hence of positive measure. 
    Hence:
    \begin{equation}
    	\frac{1}{T}\int\nolimits_0^T \chi_{a_tu_0x_0 \in LAUx_0}\bd t\xrightarrow{T\to \infty} m_{X_n}(LAUx_0) > 0
    \end{equation}
	and hence for some $t>0$ we have $a_tl_1x_1 \in LAUx_0$. 
    Since $AL=LA$ we deduce that $x_1\in LAUx_0$.

    % Write $x_i = g_i \ZZ^d$ for $i=1,2$.
    % We need to prove that $g_1^{-1}ULg_0\cap \on{SL}_d(\ZZ)$ is nonempty. 
    % Indeed, $\on{SL}_d(\RR)\setminus UL$ is an algebraic subvariety of $\on{SL}_d(\RR)$ by Bruhat decomposition, and hence $UL$ is zariski open in $\on{SL}_d(\RR)$, and so is $g_1^{-1}ULg_0$. 
    % By Borel density theorem, $g_1^{-1}ULg_0\cap \on{SL}_d(\ZZ)\neq \emptyset$, as desired.
    % Let $L'$ denote the set of lower diagonal matrices with $1$ on the diagonal. Note that $UL'=UL$, and hence for every $x_0,x_1\in X_n$ there exists $l'\in L', u\in U$ such that $l'x_0= ux_1$. 

    Let $\|\cdot\|_A: A\to[0,\infty), \|\cdot\|_L: L\to [0,\infty), \|\cdot\|_U: U\to [0,\infty)$ be the distance from the identity. 
    Define by $f: \cK^2\to [0,\infty)$ the map attaching \[f(x_0, x_1)=\inf\{\max(\|a\|_A, \|l\|_L, \|u\|_U): a\in A, l\in L, u\in U: x_1 = laux_0\}.\]
    Since $LAU$ is open and the multiplication map $L\times A\times U\to LAU$ is a homeomorphism, if we have $laux_0=x_1$ for $a\in A, l\in L, u\in U$, then for every sufficiently close $\hat x_0\sim x_0, \hat x_1\sim x_1$ there are $\hat a\in a, \hat l\in L, \hat u\in U$ arbitrarily close to $a, l,u$ such that $\hat l\hat a\hat u \hat x_0=\hat x_1$. This implies that $f$ is upper semicontinuous. 
    In particular, for every $(x_0,x_1)\in \cK^2$ we get that $f$ is bounded in a neighborhood of $(x_1,x_2)$. Since $\cK^2$ is compact this implies that $f$ is bounded everywhere, as desired.
\end{proof}
\begin{definition}
For $v,v'\in \RR^{n-1}_0$, we say that $v\preceq v'$ if for all $i=1,\dots,n-1$ we have 
    $v_{i+1}-v_i \le v_{i+1}'-v_i'$. 
    For every $v_0\in \RR^{n-1}_0$ with $0 \preceq v_0$, define a \emph{box}
    $S_{v_0} = \{v\in \RR^{n-1}_0: 0\preceq v\preceq v_0\}$.
    A \emph{boxed map} is a map $f: S_{v_0}\to X_n$ of the form $f(v) = \exp(v).x_0$ for some $x_0\in X_n$. 
    For every compact set $\cK\subset X_n$, a boxed map $f: S_{v_0}\to X_n$ is said to be \emph{$\cK$-bounded} if $f(0),f(v_0)\in \cK$. 
    Recall that $w_0 = (\frac{2i-n-1}{2})_{i=1}^n \succeq 0$. 
\end{definition}
% In the following lemma, we fix $\cK, \cK'$ and let $R\to \infty$. The constants in the $O$-notations depend on $\cK, \cK'$ and hold as $R\to \infty$. 
\begin{corollary}[Gluing boxes]\label{cor: Gluing boxes}
    Let $\cK\subset X_n$ be compact and $\cK \subset \cK' \subset X_n$
    a compact neighborhood. 
    In this corollary, all $O$ notations depend on $\cK, \cK'$. 
    Let $R>0$ be large enough as a function of $\cK, \cK'$. Fix any two $\cK$-bounded boxed maps $f_1: S_{v_1}\to X_n, f_2: S_{v_2}\to X_n$ with $v_1,v_2 \succeq Rw_0$. 
    Then there exists a $\cK'$-bounded boxed map $f: S_{v_3}\to X_n$ with 
    \begin{enumerate}[label=\textit{\emph{(\textit{\arabic*})}}]
        \item $v_3 = v_1+ v_2+O_{\cK,\cK'}(1)$;
        % \item $f$ is $\cK'$-bounded;
        \item There are two points $w_1, w_2\in S_{v_3}$
        with $\|w_1\| = o_R(R)$, $\|w_2 - v_1\| = o_R(R)$ 
        and $0<\rho = o_R(R)$ we have that for all $i=1,2$ and $v\in S_{v_i-2\rho w_0}$, 
        \begin{align}\label{eq: big box approx small}
            d_{X_n}\left( 
                f_i(\rho w_0 + v),
                f(w_i + v)
                \right) = o_R(1)
            .
        \end{align}
    \end{enumerate}
    See Figure \ref{fig: boxes} for a visualization of these conditions.
\end{corollary}
\definecolor{zzffff}{rgb}{0.6,1.,1.}
\definecolor{uuuuuu}{rgb}{0.26666666666666666,0.26666666666666666,0.26666666666666666}
\definecolor{zzttqq}{rgb}{0.6,0.2,0.}
\definecolor{ududff}{rgb}{0.30196078431372547,0.30196078431372547,1.}
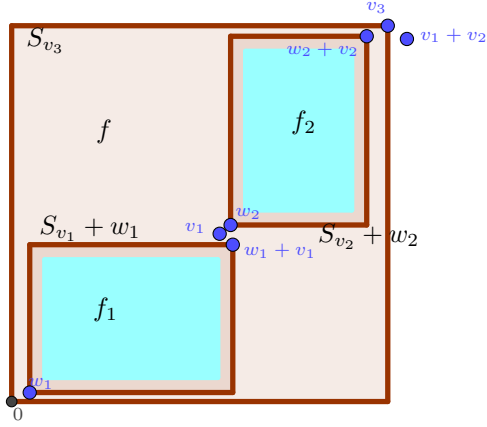
\begin{figure}
    \begin{center}
        \begin{tikzpicture}[line cap=round,line join=round,>=triangle 45,x=1.0cm,y=1.0cm]
            \clip(-0.5095730225518396,-0.3874071194705439) rectangle (6.764163095569009,5.528964283061113);
            \fill[line width=2.pt,color=zzttqq,fill=zzttqq,fill opacity=0.10000000149011612] (0.24,2.0867479281250763) -- (2.9395656192824715,2.0867479281250763) -- (2.9395656192824715,0.12) -- (0.24,0.12) -- cycle;
            \fill[line width=2.pt,color=zzttqq,fill=zzttqq,fill opacity=0.10000000149011612] (0.,0.) -- (0.,5.) -- (5.,5.) -- (5.,0.) -- cycle;
            \fill[line width=2.pt,color=zzttqq,fill=zzttqq,fill opacity=0.10000000149011612] (2.9104470128270585,4.86) -- (4.72,4.86) -- (4.72,2.3488153862237944) -- (2.9104470128270585,2.3488153862237944) -- cycle;
            \fill[line width=2.pt,color=zzffff,fill=zzffff,fill opacity=1.0] (0.44,1.8867479281250763) -- (0.44,0.32) -- (2.7395656192824713,0.32) -- (2.7395656192824713,1.8867479281250763) -- cycle;
            \fill[line width=2.pt,color=zzffff,fill=zzffff,fill opacity=1.0] (4.52,2.5488153862237946) -- (4.52,4.66) -- (3.1104470128270587,4.66) -- (3.1104470128270587,2.5488153862237946) -- cycle;
            \draw [line width=2.pt,color=zzttqq] (0.24,2.0867479281250763)-- (2.9395656192824715,2.0867479281250763);
            \draw [line width=2.pt,color=zzttqq] (2.9395656192824715,2.0867479281250763)-- (2.9395656192824715,0.12);
            \draw [line width=2.pt,color=zzttqq] (2.9395656192824715,0.12)-- (0.24,0.12);
            \draw [line width=2.pt,color=zzttqq] (0.24,0.12)-- (0.24,2.0867479281250763);
            \draw [line width=2.pt,color=zzttqq] (0.,0.)-- (0.,5.);
            \draw [line width=2.pt,color=zzttqq] (0.,5.)-- (5.,5.);
            \draw [line width=2.pt,color=zzttqq] (5.,5.)-- (5.,0.);
            \draw [line width=2.pt,color=zzttqq] (5.,0.)-- (0.,0.);
            \draw [line width=2.pt,color=zzttqq] (2.9104470128270585,4.86)-- (4.72,4.86);
            \draw [line width=2.pt,color=zzttqq] (4.72,4.86)-- (4.72,2.3488153862237944);
            \draw [line width=2.pt,color=zzttqq] (4.72,2.3488153862237944)-- (2.9104470128270585,2.3488153862237944);
            \draw [line width=2.pt,color=zzttqq] (2.9104470128270585,2.3488153862237944)-- (2.9104470128270585,4.86);
            \draw [line width=2.pt,color=zzffff] (0.44,1.8867479281250763)-- (0.44,0.32);
            \draw [line width=2.pt,color=zzffff] (0.44,0.32)-- (2.7395656192824713,0.32);
            \draw [line width=2.pt,color=zzffff] (2.7395656192824713,0.32)-- (2.7395656192824713,1.8867479281250763);
            \draw [line width=2.pt,color=zzffff] (2.7395656192824713,1.8867479281250763)-- (0.44,1.8867479281250763);
            \draw [line width=2.pt,color=zzffff] (4.52,2.5488153862237946)-- (4.52,4.66);
            \draw [line width=2.pt,color=zzffff] (4.52,4.66)-- (3.1104470128270587,4.66);
            \draw [line width=2.pt,color=zzffff] (3.1104470128270587,4.66)-- (3.1104470128270587,2.5488153862237946);
            \draw [line width=2.pt,color=zzffff] (3.1104470128270587,2.5488153862237946)-- (4.52,2.5488153862237946);
            \draw (0.9595930304258272,1.5053023001313062) node[anchor=north west] {$f_1$};
            \draw (3.5935033416200226,3.9936195790484237) node[anchor=north west] {$f_2$};
            \draw (0.9993002210468452,3.87449800718537) node[anchor=north west] {$f$};
            \draw (0.24486359924750284,2.590632177105794) node[anchor=north west] {$S_{v_1}+w_1$};
            \draw (3.95, 2.484746335449746) node[anchor=north west] {$S_{v_2} + w_2$};
            \draw (0.07279910655642474,5.0921851862299174) node[anchor=north west] {$S_{v_3}$};
            \begin{scriptsize}
                \draw [fill=ududff] (2.9395656192824715,2.0867479281250763) circle (2.5pt);
                \draw[color=ududff] (3.0905455937537942,2.001642182894029) node {$~~~~~~~~~~w_1+v_1$};
                \draw [fill=ududff] (0.24,0.12) circle (2.5pt);
                \draw[color=ududff] (0.39045663152456894,0.24129006536223846) node {$w_1$};
                \draw [fill=ududff] (4.72,4.86) circle (2.5pt);
                \draw[color=ududff] (4.559711646731461,4.688495414916236) node {$w_2+v_2~~~~~~~~~$};
                \draw [fill=ududff] (2.9104470128270585,2.3488153862237944) circle (2.5pt);
                \draw[color=ududff] (3.050838403132776,2.5045999307602553) node {$~~w_2$};
                \draw [fill=uuuuuu] (0.,0.) circle (2.0pt);
                \draw[color=uuuuuu] (0.08603483676343075,-0.16901757105494586) node {$0$};
                \draw [fill=ududff] (5.,5.) circle (2.5pt);
                \draw[color=ududff] (4.850897711285594,5.244396083610486) node {$v_3$};
                \draw [fill=ududff] (2.7648539805499923,2.232340960402142) circle (2.5pt);
                \draw[color=ududff] (2.4419948136105,2.306063977655166) node {$v_1$};
                \draw [fill=ududff] (5.25449483248782,4.823896934933907) circle (2.5pt);
                \draw[color=ududff] (5.499448491428888,4.873795637814319) node {$~~~~~~~~v_1+v_2$};
            \end{scriptsize}
        \end{tikzpicture}
    \end{center}
    \caption{Plot of the boxes in Corollary \ref{cor: Gluing boxes}. One can see $S_{v_3}$, and in it $S_{v_1}+w_1, S_{v_2}+w_2$, and in them the regions approximated by $f_1, f_2$. }
    \label{fig: boxes}
\end{figure}
\begin{proof}
    Let $\cK_A\subset A, \cK_L \subset L, \cK_U\subset U$ as in Lemma \ref{lem: diagonal closing lemma}, constructed for $\cK$.
    Let $C_0>0$ be large enough so that $\cK_A = \exp(\cK_0)$ for some compact 
    \[\cK_0\subseteq \{w\in \RR^{n-1}_0:-C_0 w_0 \preceq w\preceq C_0w_0\} = S_{2C_0w_0} - C_0w_0.\]
    Note that 
    \begin{align}\label{eq: L expanding}
        \text{for all~~} T>0, v \succeq Tw_0, k_L\in \cK_L\text{~~we have~~}d_{\on{SL}_n(\RR)}(\exp(-v)k_L\exp(v), I) = O(\exp(-T)),
    \end{align}
    and similarly, 
    \[
        \text{for all~~} T>0, v \succeq Tw_0, k_U\in \cK_U\text{~~we have~~}d_{\on{SL}_n(\RR)}(\exp(v)k_U\exp(-v), I) = O(\exp(-T)). 
    \]
    
    Apply Lemma \ref{lem: diagonal closing lemma} for $f_1(v_1), f_2(0)$, we get that there are 
    $v_0\in \RR^{n-1}_0, k_L\in \cK_L, k_U\in \cK_U$ such that \[|(v_0)_i - (v_0)_{i+1}| < C_0\text{ for all }i=1,\dots,n-1,\] and 
    \[k_L\exp(v_0)k_Uf_2(0) = f_1(v_1).\] 
    Let $\rho = \sqrt{R}$, $v_3 = v_1 + v_2 + v_0, w_1 = 0, w_2 = v_1+v_0$, $x_0 = \exp(-v_1) k_L\exp(v_1) f_1(0)$, $f: S_{v_3}\to X_n$ defined by 
    $f(v) = \exp(v)x_0$ for all $v\in S_{v_3}$.
    % By Eq. \eqref{eq: L expanding} we get that 
    % $d(\exp(-v_1) k_L\exp(v_1), I) = O(\exp(-R)) = o_R(1)$. Hence we deduce that for $R$ sufficiently big, $x_0 \in \cK'$. 
    Let $v \in S_{v_1 - \rho w_0}$. By the definition of $f$,
    \begin{align*}
        f(v) &= \exp(v)x_0 = \exp(v - v_1) k_L\exp(v_1 - v) f_1(v).
    \end{align*}
    Since $v\preceq v_1 -  \rho w_0$ we get that $v_1 - v \succeq \rho w_0$, and hence by Eq. \eqref{eq: L expanding},
    \[
        d_{\on{SL}_n(\RR)}\left( \exp(v - v_1) k_L \exp(v_1 - v), I \right) = o_R(1),
    \]
    which implies 
    $d_{\on{SL}_n(\RR)}(f(v), f_1(v)) = o_R(1)$. 
    Eq. \eqref{eq: big box approx small} for $i=1$. Substituting $v=0$ yields 
    $d_{\on{SL}_n(\RR)}(f(0),  f_1(0)) = o_R(1)$, which implies that $f(0) \in \cK'$ provided that $R$ is large enough as a function of $\cK,\cK'$. 
    Showing \eqref{eq: big box approx small} for $i=2$ and $f(v_3) \in \cK'$ is done similarly. 

    % $d(f(v), f_1(\rho w_0+v)) = o_R(1)$. 

    % Note that 
    % \[f(v_3) = \exp(v_1+v_2+v_0) x_0 = \exp(v_2) \exp(v_0) k_U f_1(v_1)  = \exp(v_2) k_{l}^{-1} f_2(0) = \exp(v_2) k_{l}^{-1} \exp(-v_2) f_2(v_2).\]
    % From here the proof of Eq. \ewref{eq: big box approx small} for $i=2$ is symmetric to the approximation of $f_1$ and $f_1(0)$.
\end{proof}
\begin{corollary}\label{cor: boxes approximation}
	For any $\nu_1, \dots, \nu_k$ ergodic $A$-invariant measures on $X_n$, $c_1,c_2,\dots,c_k > 0$ with $c_1+\dots+c_k = 1$ there exists a collection of lattices $(x_R)_{R>0}\in X_n$ such that for every $i=1,\dots, k$:
	\begin{equation}\label{eq: 18:51}
		\lim_{R\rightarrow \infty}\frac{1}{\on{vol}(S_{R, i})}\left(v\mapsto \exp\left(v\right).x_R\right)_*m_{\RR^{n-1}_0}\mid_{S_{R, i}}=\nu_i,
	\end{equation}
    where $S_{R, i} = (c_1+\dots+c_{i-1})Rw_0 + S_{c_iR w_0}$.
	% where $B_i(R)=\frac{iR}{k}w_0 + S_{\frac{R}{k}w_0}$ for $i=1,\dots,k$.
    In addition, the set $\{x_R: R>0\}$ is pre-compact.
\end{corollary}
\begin{proof}
    Let $\cK$ be a compact set with $\nu_i(\cK)>0$ for each $i$, and for each $i=1,\dots,k$ choose $y_i\in \cK$ to be a generic point for the $A$ action on $\nu_i$, with the F\"olner sequence $\{S_{Rw_0}:R>0\}$.
    By ergodicity of $\nu_i$:
    \[\lim_{R\to \infty} \frac{1}{\on{vol}(S_{c_iRw_0})}\int_{S_{c_iRw_0}}\chi_{\exp(v)y_i\in \cK} \bd m_{\RR^{n-1}_0} = \nu_i(\cK) > 0.\]
    Denote \[\rho_i = \sup\{\rho>0: \exp(v)y_i \notin \cK, \forall v\in c_iRw_0 - S_{\rho w_0}\}.\]
    Since $\int_{c_iRw_0 - S_{\rho w_0}}\chi_{\exp(v)y_i\in \cK}\bd m_{\RR^{n-1}_0} = 0$,
    we get that $\frac{\rho_i}{R}\xrightarrow{R\to \infty}0$.
    By the Definition of $\rho_i$, there is $v_{R, i}\in c_iRw_0 - S_{\rho_iw_0}$ such that $\exp(v_{R, i})y_i \in \cK$. 
    Since $\abs{\|v_{R, i}\|-c_iR} = O(\rho_i) = o_R(R)$ we deduce that 
    \[\frac{1}{\on{vol}(S_{v_{R, i}})}\left(v\mapsto \exp(v)y_i\right)_*m_{\RR^{n-1}_0}\mid_{S_{v_{R, i}}}\xrightarrow{R\to \infty}\nu_i.\]
    The corollary follows now from iteratively applying Corollary \ref{cor: Gluing boxes} to glue the boxed maps $(v\mapsto \exp(v)y_i)|_{S_{v_{R, i}}}$ to one boxed map $f$, using an increasing set of compact neighborhoods $\cK\subset \cK_1\subset \cK_2\subset\dots\subset \cK_k$. 
\end{proof}
As a corollary of Lemma \ref{lem: fastApproximation2} and Corollary \ref{cor: boxes approximation} we can prove Theorem \ref{thm: ergodicDecomp}:
\begin{proof}[Proof of Theorem \ref{thm: ergodicDecomp}]
    To prove Theorem \ref{thm: ergodicDecomp}, together with the bounds predicted in Remark \ref{rem: uniformRN}, it is sufficient to prove them with finitely many measures. 
    Let $\nu_1,\nu_2,\dots, \nu_k$ be ergodic $A$-invariant measures on $X_n$, and $(c_i)_{i=1}^k$ positive numbers with $c_1+\dots+c_k=1$. 
    For any $R>0$, let $x_R\in X_n$ be the point guaranteed by Corollary \ref{cor: boxes approximation}. 
    Let $\cK$ the compact closure of $\{x_R: R>0\}$.

    Lemma \ref{lem: fastApproximation2} 
    guarantees a sequence $(t_m)_{m=0}^\infty$ such that $t_m \xrightarrow{m\to \infty}\infty$ such that for all $R>0, m\ge 0$ there is $y_{R,m} \in X_n$ such that $d_{X_n}(x_R, y_{R,m}) = O(e^{-\alpha t_m}),$
    where $\alpha = \frac{\lfloor n/2\rfloor}{4n(n^2-1)}>0$, and
    $y_{R,m}$
    is stabilized by a subgroup $\exp(\Lambda_{R, m})\subseteq A$ where $\Lambda_{R, m}\subseteq \RR^{n-1}_0$ is generated by vectors $(v^{(i)}_{R,m})_{i=1}^{n-1}$ such that $\left|\left( v_{R,m} ^ {(i)} \right)_j - (1 - n\delta_{ij}) t_m\right| = O(1)$
    % \yinote{there was a mistake here, I added $n$ before $\delta_{i,j}$. This complicates things a bit because now it is not a rectangle anymore. But it doesn't change the idea of the argument. I will fix the rest later.}.
    Choose $R_m = \frac{\alpha t_m}{n-1} - \sqrt{t_m}$ so that all $v\in S_{R_mw_0}$ we have $d_{X_n}(\exp(v)x_{R_m}, \exp(v)y_{R_m,m}) = o_m(1)$. 
    This implies that 
    \begin{equation}\label{eq: chunks in Ay close to mu_i}
		\lim_{m\rightarrow \infty}\frac{1}{\on{vol}(S_{R_m,i})}\left(v\mapsto \exp\left(v\right).y_{R_m,m}\right)_*m_{\RR^{n-1}_0}\mid_{S_{R_m,i}}=\nu_i,
	\end{equation}
    where $S_{R_m,i}$ are as in Corollary \ref{cor: boxes approximation}.
    The choice of $R_m$ guarantees that $S_{Rw_0}$ injects into $\RR^{n-1}_0/\Lambda_{R_m,m}$ and hence Eq. \eqref{eq: chunks in Ay close to mu_i} implies that 
    \begin{align*}
        \lim_{m\to \infty} \mu_{A y_{R_m,m}}\ge
        \mu \cdot \lim_{m\to \infty}\frac{\on{vol}(S_{R_m,i})}{\cov{\Lambda_{R_m,m}}},
    \end{align*}
    and explicit computation shows that 
    \[
        \lim_{m\to \infty}\frac{\on{vol}(S_{R_m,i})}{\cov{\Lambda_{R_m,m}}} = \left(\frac{\lfloor n/2\rfloor}{4n^2(n-1)^2(n+1)}\right)^{n-1} c_i^{n-1}.
    \]
    % and 
    % \begin{align*}
    %     \frac{\on{vol}(S_{R_m,i})}{\cov{\Lambda_{R_m,m}}}
    %     \asymp \frac{c_i^{n-1}R_m^{n-1}}{R_m^{n-1}} = c_i^{n-1}.
    % \end{align*}
%     we can find a sequence $R_{m}\rightarrow \infty$ and $\alpha\in (0,1)$ such that there exists $(x_m)_m\subset X_n$ and $v_1(m),\dots,v_{n-1}(m)\in \RR^{n-1}_0$ such that:
% 	\begin{equation}
% 		\stab_{\RR^{n-1}_0}(x_m)=\on{span}_{\ZZ}\{v_1(m),\dots,v_{n-1}(m)\};
% 	\end{equation}
% 	\begin{equation}
% 		\norm{v_i(m)-R_me_i}\leq Ce^{-m}\text{ for all }i=1,\dots,n-1;
% 	\end{equation}
% 	\begin{equation}
% 		d(\exp(v).x_m,y_{R_m})\leq Ce^{-m}\text{ for all }v\in B^{\infty}_{\alpha R_m}(0).
% 	\end{equation}
% By Equation (\eqref{eq: 18:51}) we deduce that any weak limit $\mu$ of $\mu_{Ax_m}$ will satisfy $\mu=\frac{\alpha}{k}\left(\sum_{i=1}^k\mu_i\right)+\mu_0$ for some invariant positive measure $\mu_0$ satisfying $\mu_0(X_n)\leq 1-\alpha$.
\end{proof}

\section{Entire Mass Approximations}
\label{sec: entire mass aprox}
First, we prove Corollary \ref{cor: tailoredEscape} of Theorem \ref{thm: ergodicDecomp} and \cite[Theorem 1.1]{ShapiraEscape}.
\begin{proof}[Proof of Corollary \ref{cor: tailoredEscape}]
	Let $\nu_1,\dots,\nu_k,\ldots\in \einv$. By ~\cite[Theorem 1.1]{ShapiraEscape} we can find 
    $(\rho_m)_m\subset \cinv$ such that $\rho_m\rightarrow 0$ weakly. 
    Let $c_0, c_1,c_2,\ldots$ be a sequence of positive real numbers with $\sum_{i=0}^\infty c_i = 1$.
    By Remark \ref{rem: uniformRN} of Theorem \ref{thm: ergodicDecomp}, for any $m$ there exists $\mu^{(m)}\in \overline \cinv$ such that $\mu^{(m)} \ge \Theta(c_0^{n-1})\rho_m, \Theta(c_i^{n-1})\nu_i$ for each $i\ge 1$.
    Any weak limit $\mu^{(\infty)}$ of $(\mu^{(m)})_m$ will belong to $\overline \cinv$ and satisfy, since $\rho_m\rightarrow 0$, that at least $\Theta(c_0^{n-1})$ of the mass of $\mu^{(m)}$ escapes, namely $\mu^{(\infty)}(X_n)\leq 1-\Theta(c_0^{n-1})$. On the other hand, for all $i\ge 1$, $\nu_i$ will still appear in the ergodic decomposition of $\mu^{(\infty)}$ with coefficient bounded below by $\Theta(c_i^{n-1})$ as desired.
\end{proof}
The remainder of this section will be dedicated to the proof of Theorem \ref{thm: haarApproximation}. The proof of this theorem will follow similar lines as the proof of Theorem \ref{thm: ergodicDecomp}: We first introduce the desired results on Hecke orbits, then apply it to the special number fields from \cite{ShapiraEscape}% and prove some result on its units. Then we will connect these results to Hecke operators and use the properties of Hecke operators 
 to complete the proof.
% In contrast to the method in the previous section, here we will fix $p$ to be the minimal prime which is congruent to $1$ mod $2n$. 
\subsection{Distribution of Hecke neighbors of points close to the cusp}

The following corollary analyzes the distribution properties of Hecke operators where we begin with a lattice high up in the cusp.
For that, we will use the qualitative version of Theorem \ref{thm: equidistributionOfHecke} (See also \cite[Th{\'e}or{\`e}me 1.2]{clozel2004equidistribution}). 
% but appeared first at \cite[Th{\'e}or{\`e}me 1.2]{clozel2004equidistribution}.
% \begin{corollary}\label{cor: equidistribution Of Hecke weak}
%     Fix $x\in X_n$, and a prime number $p$. For $k_1\le k_2\le \dots\le k_n, a_{p; k_1, k_2, \dots, k_n}\in \on{SL}_n(\RR)$ as in Definition \ref{def: Hecke} we have:
%     \begin{equation}
%             T_{a_{p; k_1, k_2, \dots, k_n}}^{\rm M}(x)\xrightarrow{k_n-k_1\to\infty} m_{X_n}. 
%     \end{equation}
%     % where $\alpha>0$ is some constant depending only on $n$. 
% \end{corollary}
\begin{corollary}\label{cor: equidistribution Of Hecke weak}
    Fix $x\in X_n$. For a prime $p$, $k_1\le k_2\le \dots\le k_n, a_{p; k_1, k_2, \dots, k_n}\in \on{SL}_n(\RR)$ as in Definition \ref{def: Hecke} we have:
    \begin{equation}
            T_{a_{p; k_1, k_2, \dots, k_n}}^{\rm M}(x)\xrightarrow{(k_n-k_1)\log p\to\infty} m_{X_n}. 
    \end{equation}
    % where $\alpha>0$ is some constant depending only on $n$. 
\end{corollary}
\begin{remark}
    The convergence is uniform on compact sets. This follows from the fact that $T_a(gx) = gT_a(x)$ for all $g\in \on{SL}_n(\RR)$, and the result for the particular case $x=\ZZ^n\in X_n$. 
\end{remark}
Recall the definition of Minkowski's successive minima $\lambda_i(x)$ for $x\in X_n, i=1,\dots,n$ from \cite[Chapter VIII]{cassels2012introduction}. 
They satisfy that a closed set $U\subseteq X_n$ is compact if and only if $\inf_{x\in U}\lambda_1(x)>0$ if and only if $\sup_{x\in U} \lambda_n(x) < \infty$.

\begin{theorem}\label{thm: Hecke in the cusp}
    For every prime number $p$, denote $a_p = a_{p; 0,1,1,\dots,1}$ as in Definition \ref{def: Hecke}. 
    For every point $x\in X_n$ and $x'\in T_{a_p}(x)$ we have
    \begin{align}\label{eq: stays in the cusp}
        p^{1/n}\lambda_{1}(x) \ge \lambda_{1}(x')
    \end{align}
    % and $p^{k(n-1)/n}\lambda_1(x) \ge \lambda_1(x')$. 
    In addition,
    for every $\delta>0$ there exists a compact set $\cK(\delta)$ such that the following holds.
    Let $(x_i)_{i=1}^\infty$ be a sequence of lattices and let $p_i\xrightarrow{i\to \infty}\infty$ be a prime sequence. 
    Assume that 
    \begin{align}\label{eq: condition of nondivergence}
      p_i^{\frac{1}{n}}\lambda_1(x_i) \ge 1.
    \end{align}
    Then 
    \begin{align}\label{eq: nondivergence of T_{k_i}(x_i)}
        \liminf_{i\to \infty} T^{\rm M}_{a_{p_i}}(x_i)(\cK(\delta)) > 1-\delta.
    \end{align}
\end{theorem}
\begin{remark}\label{rem: can do equidistribution}
    Although we state no-escape-of-mass in Eq. \eqref{eq: nondivergence of T_{k_i}(x_i)}, one can prove an equidistribution result:
    If we replace Eq. \eqref{eq: condition of nondivergence} by 
    \begin{align}\label{eq: condition of nondivergence2}
      p_i^{\frac{1-\varepsilon}{n}}\lambda_1(x_i) \ge 1.
    \end{align}
    for some $\varepsilon>0$ we get that
    \begin{align*}
        T^{\rm M}_{a_{p_i}}(x_i)\xrightarrow{i\to \infty} m_{X_n}.
    \end{align*}
    This result could simplify the proof of Theorem \ref{thm: haarApproximation}, but its proof is too complicated and can be avoided.
    % The proof uses the fact that if $k=k'+s$ then $T_{a^k}^{\rm M}$ is a large component of $T^{\rm M}_{a^{k'}}\circ T^{\rm M}_{a^s}$, and a $p$-adic interpretation of Hecke operators. 
\end{remark}
\begin{remark}\label{rem: extension for other hecke}
  With Eq. \eqref{eq: condition of nondivergence2} instead of \eqref{eq: condition of nondivergence}, Theorem \ref{thm: Hecke in the cusp} could be extended for the Hecke operators $T_{a_{p_i}^{k_i}}$, with $p_i^{k_i}$ instead of $p_i$. 
  For different Hecke operators, there are other thresholds, stated in terms of different Minkowski successive minima. The equidistribution result holds as well.
\end{remark}
\begin{proof}[Proof of Theorem \ref{thm: Hecke in the cusp}]
  The bound on $\lambda_1$ for every Hecke neighbor follows from the definition of Hecke operators.
  Since $x\subset p^{-1/n}x'$ we get $\lambda_{1}(x') \le p^{1/n}\lambda_{1}(x)$.
  % The index $p$ inclusion $p^{k/n}x'\subseteq x$ implies that $px\subseteq p^{k/n}x'$, and hence $p^{k(n-1)/n}\lambda_1(x)\ge \lambda_1(x')$. 

  Let $x\in X_n$ be a lattice with $\lambda_1(x) \ge p^{-\frac{1}{n}}$ and $x'$ be a random point in $T_{a_p}(x)$. 
  We show that for every $\delta>0$ there is a compact set $\cK\subset X_n$ depending only on $\delta$ such that 
  $\PP(x'\in \cK) \ge 1-\delta$ provided that $p$ is sufficiently large as a function of $\delta$. 
  % $\frac{1}{\#T_{a^{k}}(x)} \sum_{x'\in T_{a^{k}}(x)}\delta_{x'} (\cK) \ge 1-\delta$, provided that $k$ is sufficiently large as a function of $\varepsilon, \delta$. 

  Let $\cK_r = \{y\in X_n:\lambda_1(y)\ge r\}$ for some $1>r>0$ which will be chosen later as a function of $\delta$.
  To bound $\PP(x'\notin \cK_r)$, we will bound: 
  \begin{align*}%\label{eq: prob less then expectation}
      \EE(\# (x'\cap B(r)\setminus \{0\}))\ge \PP(x'\notin \cK_r). 
  \end{align*}
  Here $B(r)$ is the radius $r$ ball in $\RR^n$.
  Since $p^{(n-1)/n}x' \subseteq x$, we have:
  \begin{align}\label{eq: additivity of expectation}
      \EE(\#(x'\cap B(r)\setminus \{0\})) = \sum_{v\in B(r)\cap p^{-(n-1)/n}x\setminus \{0\}}\PP(v\in x').
  \end{align}
  To analyze the probability $\PP(v\in x')$ note that  $px \subseteq p^{(n-1)/n}x'$ and satisfies
  $p^{(n-1)/n}x'/px = \ZZ/p$. 
  In particular, if $v\in p^{1/n}x$ then $\PP(v\in x') = 1$. 
  Otherwise, if $v\in (p^{-(n-1)/n}x\setminus p^{1/n}x)\cap x'$ then $v$ determines $x'$ by the formula $p^{(n-1)/n}x' = px + w\ZZ$, and hence 
  \[\PP(v\in x') = \frac{1}{\#T_{a_p}(x)} = \frac{p-1}{p^n-1}<\frac{1}{p^{n-1}}.\]
  Using these estimates we get:
  \begin{align}\label{eq: prob bound by division by p}
    \sum_{v\in B(r)\cap p^{-(n-1)/n}x\setminus \{0\}}\PP(v\in x') 
    \le \sum_{v\in B(r)\cap p^{-(n-1)/n}x\setminus p^{1/n} x}\frac{1}{p^{n-1}} + \sum_{v\in B(r)\cap p^{1/n}x\setminus \{0\}} 1\\
    \nonumber
    \le \frac{\#(B(p^{(n-1)/n}r)\cap x)}{p^{n-1}} + \#(B(p^{-1/n}r)\cap x)
    .
  \end{align}
  The assumption on $\lambda_1(x)$ implies that the second term in the right-hand side of Eq. \eqref{eq: prob bound by division by p} vanishes.
  To bound the first term, we need to analyze $\#B(p^{(n-1)/n}r) \cap x$. 
  % Simple volume computation shows that for all $R>\lambda_1(x)$ we have $\#B(R) \cap x \asymp (R/\lambda_1(X))^{n}$. 
  The following claim is a different wording of \cite[Lemma 3.5]{LSST}. 
  \begin{claim}
    For every lattice $x$, $R > \lambda_1(x)$ we have 
    \[\#\left( B(R)\cap x\setminus 0 \right)  \asymp \max_{i=1}^n \frac{R^i}{\lambda_1(x)\cdots \lambda_i(x)}\]
  \end{claim}
  By Minkowski's theorem (\cite[Chapter VIII, Theorem V]{cassels2012introduction}), $\lambda_1(x)\cdots \lambda_n(x) \asymp \cov{x}$.
  In particular, $\lambda_n(x) \ll p^{(n-1)/n}$. 
  Thus, 
  \begin{align}\label{eq: ball estimate}
    \#B(p^{(n-1)/n}r) \cap x &\le 
    \max_{i=1}^n \frac{\left( p^{(n-1)/n}r \right)^i}{\lambda_1(x)\cdots \lambda_i(x)} = 
    \prod_{i=1}^n \max\left( \frac{p^{(n-1)/n}r}{\lambda_i(x)}, 1\right)\\\nonumber&
    \asymp
    \prod_{i=1}^n \max\left( p^{(n-1)/n}r, \lambda_i(x)\right)
    \ll 
    p^{(n-1)/n}r \cdot (p^{(n-1)/n})^{n-1} = rp^{n-1},
  \end{align}
  provided that $r > \frac{C}{p^{(n-1)/n}}$ where $C = \max_{y\in X_n}\lambda_1(y)$.

  Altogether Eqs. \eqref{eq: additivity of expectation}, \eqref{eq: prob bound by division by p} and \eqref{eq: ball estimate} imply that
  \begin{align*}
    \PP(x'\notin \cK_r) \ll r.
  \end{align*} 
  For every $\delta > 0$ choose $r(\delta)>0$ sufficiently small so that we have $\PP(x'\in \cK_{r(\delta)})>1-\delta$ for all $p\ge \delta^{-n/(n-1)}$. 
\end{proof}

\begin{definition}[Shapira's orbit]\label{def: shapira's orbit}
Let $0 < \eta < \frac{1}{2n}$ fixed and $M>0$ be a size parameter. 
Let $0\le a_1< \dots< a_n \le M$ be integers with $a_{i+1}-a_i\ge\eta M$, say, $a_i = \lfloor iM/n \rfloor$. 
Let $P(z) = (z-a_1)(z-a_2)\cdots(z-a_n) - 1$ and $K = \QQ[\alpha]$ where $P(\alpha) = 0$. 
Let $x_{\ZZ[\alpha]}$ as in Definition \ref{def: compact orbit}. 	
\end{definition}
The theorem we state is an accumulation of results and computations given in Shapira \cite{ShapiraEscape}.
\begin{definition}
  For every $\pi\in S_n$ (the permutation group), denote by $F_\pi\subseteq \RR_{0}^{n-1}$ the set of vectors 
  $F_\pi = \{v\in \RR^{n-1}_0: v_{\pi(i)}\ge v_{\pi(i+1)} - 1: i=1,\dots,n\}$. 
  Here we use cyclic index notations and $\pi(n+1) = \pi(1)$. 
  The set $F_\pi$ is compact. 
  Note that $F_\pi$ is well defined for $\pi\in S_n/C_n$ where $C_n$ is the cyclic group of rotations. 
\end{definition}
We recall a result by Shapira, which analyzes a family of compact orbits and shows that they approximate in most of their volume parts of the orbit $A\ZZ^n$. 
\begin{theorem}[Shapira \cite{ShapiraEscape}]\label{thm: Shapira escape}
  In this theorem the constant in the $O$-notation depends on $\eta$. 
  There is a bound $M_0(\eta)$ such that for all $M>M_0(\eta)$ the following occurs.
  % Let $\eta>0$ be a fixed constant, then for all $M>M_0(\eta)$ the following holds. 
  % Let $K = \QQ(\alpha)$ such that $P(\alpha)=0$ where $P(z) = \prod_{i=1}^n(z-a_i) - 1$ for $a_1<a_2<\dots<a_n$ with $|a_i|<M$ for all $i=1,\dots,n$ and 
  % $a_i+\eta M < a_{i+1}$ for all $i=1,\dots,n-1$. 
  Let $x_{\ZZ[\alpha]}$ as constructed above. It is stabilized by $\exp\Lambda$ where $\Lambda$ is generated by $v_1,\dots,v_{n-1}, v_n$ satisfying 
  $v_1+\dots+v_n = 0$ and 
  \[\left|\left( v^{(i)} \right)_j - (2 - 2n\delta_{ij})\log M\right| = O(1).\]
  Moreover there is a finite collection of points $P_0 \subseteq \RR^{n-1}_0/\Lambda$ and a map $\pi: P_0\to S_n / C_n$ such that
  \[\frac{\on{vol}\left(\bigsqcup_{p\in P_0}p + (1-o_M(1))\log (M)F_{\pi(p)}\right)}{\on{vol}(\RR^{n-1}_0 / \Lambda)} = 1-o_M(1),\]
  and for all $p\in P_0, v\in (1-o_M(1))\log (M)F_{\pi}$ one has 
  \[d_{X_n}(\exp(p+v)x_{\ZZ[\alpha]}, \exp(v)\ZZ^n) = o_M(1).\]
\end{theorem}
Denote by $\on{min-co}:\RR^{n-1}_0\to (-\infty, 0]$ the minimum of the coordinates function, and note that $\min_{F_\pi}(\on{min-co}) = -\frac{n-1}{2}$. 
This function is important since 
\begin{align}\label{eq: lambda n to max_coef}
    \lambda_1(\exp(v)\ZZ^n) = \exp(\on{min-co}(v))\quad\text{for all}\quad v\in \RR^{n-1}_0. 
\end{align}
Let $m_{F_\pi}$ be the uniform probability measure on $F_\pi$ for all $\pi\in S_n$. 
Note that $(\on{min-co})_*m_{F_\pi}$ is independent of $\pi\in S_n$. 
We use the following corollary of Shapira's result, which follows from Eq. \eqref{eq: lambda n to max_coef}:
\begin{corollary}\label{cor: quantitive lambda n escape}
    In the setting of Theorem \ref{thm: Shapira escape}, and for any $I=[a,b]\subset (-\infty,0]$
\begin{align*}
	\int_{X_n} \delta_{\frac{1}{\log M}\log\lambda_1(x')\in I}\bd\mu_{\ZZ[\alpha]}(x')=(\on{min-co})_*m_{F_1}(I)+o_M(1).
\end{align*}
\begin{comment}
    \begin{align*}
        d\left( \int \delta_{\frac{1}{\log M}\log\lambda_1(x')}\bd\mu_{\ZZ[\alpha]}(x'), (\on{min-co})_*m_{F_1}\right) = o_M(1).
    \end{align*}		
\end{comment}
    % \osnote{This $d$ is not defined.}
    % \yinote{Changed the formulation. We didn't use the definition of $d$ in its full power}
\end{corollary}

\begin{proof}[Proof of Theorem \ref{thm: haarApproximation}]
  Let $c\in (0,1]$ and let $M>0$ be large enough. From now on, every measure in our construction will depend on $M$.
  Since $m_{F_1}$ is absolutely continuous with respect to Lebesgue, there is $\xi > 0$ such that $m_{F_1}(\on{min-co}^{-1}[-\xi, 0]) = c$. 
  Let $p$ be the first prime number $p>M^{n\xi}$. 
  By the Prime Number Theorem, $p = M^{\xi n + O(1/\log M)}$. 
  Let $q$ be the first prime number with $q>\log M$. Again by the Prime Number Theorem , $q = M^{O(\log \log M/\log M)}$. 
  Altogether,   
  \begin{align}\label{eq: pq size}
    pq  = M^{\xi n + O(\log \log M/\log M)}.
  \end{align}
  Let $a_p, a_q$ as in Theorem \ref{thm: Hecke in the cusp} constracted for $p,q$ respectively. 
  Let $x_{\ZZ[\alpha]}=x_{\ZZ[\alpha]}(M)$ be as in Definition \ref{def: shapira's orbit}.
  The measure \[\nu_0 = T^{\rm M}_{a_q}(T^{\rm M}_{a_p}(\mu_{Ax_{\ZZ[\alpha]}})),\]
  is supported on several compact orbits which we will now describe.  
  Consider the set 
  \[B = T_{a_q}(T_{a_p}(x_{\ZZ[\alpha]}))\subset X_n. \] 
  Define a partition
  \begin{equation}\label{eq: b-partition}
  		B = B_1\sqcup B_2 \sqcup \dots\sqcup B_r
  \end{equation}
  by saying that two points $y_0, y_1\in B$ are in the same set $B_i$ if $Ay_0=Ay_1$. 
  Note that $B$ is invariant under the $\stab_{A}(x_{\ZZ[\alpha]})$, and hence so are $B_i$.
  For every $x = ax_{\ZZ[\alpha]}\in Ax_{\ZZ[\alpha]}$ define $T_{B_i}(x) = aB_i$. 
  The element $a$ is well defined up to $\stab_{A}(x_{\ZZ[\alpha]})$, and since $B_i$ is invariant to this ambiguity, $T_{B_i}(x)$ is well defined.

  Define $T_{B_i}^{\rm M}(x) = \frac{1}{\#T_{B_i}(x)}\sum_{y\in T_{B_i}(x)}\delta_y$. 
  We deduce that \[\nu_0 = \sum_{i=1}^r \beta_i T_{B_i}^{\rm M}(\mu_{x_{\ZZ[\alpha]}}),\]
  for some $\beta_i>0$ with $\sum_{i=1}^r\beta_i = 1$ (see Remark \ref{rem: beta i} for further discussion on $(\beta_i)_{i=1} ^ r$). 
  We will now analyze $\nu_0$ and $T_{B_i}^{\rm M}(\mu_{x_{\ZZ[\alpha]}})$. 

  % Choose $k$ such that $\frac{k \log p}{n \log M} = \xi(1 + o_M(1))$, and fix $\varepsilon > 0$. 

  \textbf{Analysis of Hecke-like operators on different parts of the compact orbit:}
  Fix $\varepsilon>0$ to be chosen later.
  Distinguish two parts of $Ax_{\ZZ[\alpha]}$: 
  \begin{align}
      \label{eq: ax-def}
      (Ax_{\ZZ[\alpha]})^{-} = \{x'\in Ax_{\ZZ[\alpha]}: \lambda_1(x') \le \exp((-1 - \varepsilon) \xi\log M)\}\\
      \label{eq: ax+def}
          (Ax_{\ZZ[\alpha]})^{+} = \{x'\in Ax_{\ZZ[\alpha]}: \lambda_1(x') \ge \exp((-1 + \varepsilon) \xi\log M)\}.
  \end{align}
  Then by Corollary \ref{cor: quantitive lambda n escape} we get that $\mu_{\ZZ[\alpha]}((Ax_{\ZZ[\alpha]})^{-})  = 1-c + O(\varepsilon) + o_M(1)$ and 
  $\mu_{\ZZ[\alpha]}((Ax_{\ZZ[\alpha]})^{+}) = c + O(\varepsilon) + o_M(1)$.
  Consequently, we get that 
  \begin{align}\label{eq: everyting else is negligable}
      \mu_{\ZZ[\alpha]}(Ax_{\ZZ[\alpha]} \setminus ((Ax_{\ZZ[\alpha]})^{+}\cup (Ax_{\ZZ[\alpha]})^{-})) = O(\varepsilon) + o_M(1). 
  \end{align}

  \textbf{Analysis on $(Ax_{\ZZ[\alpha]})^{-}$:}
  We will now analyze the escape of mass for the measures $T_{B_i}^{\rm M}(\mu_{x_{\ZZ[\alpha]}})$. 
  For all $x'\in (Ax_{\ZZ[\alpha]})^{-}$, Theorem \ref{thm: Hecke in the cusp} implies that $(pq)^{1/n}\lambda_1(x') \ge \lambda_1(y')$ for all $y'\in T_{a_q}(T_{a_p}(x'))$. 
  By Eqs. \eqref{eq: ax-def} and \eqref{eq: pq size}, we get an upper bound of 
  \begin{align*}%\label{eq: lower bound where +}
    \lambda_1(y') \le M^{-\xi\varepsilon + O(\log\log M / \log M)}.
  \end{align*}
  Hence if $\varepsilon$ has a sufficiently slow decay rate as a function of $M$, say, $\varepsilon = \frac{1}{\log\log M}$, we get that for all $i=1,\dots,r$,
  \begin{align}\label{eq: - behavior}
    \int_{(Ax_{\ZZ[\alpha]})^-}\int\chi_{\mathcal{K}_\varepsilon}\bd(T^{\rm M}_{B_i}(x')) \bd\mu_{\ZZ[\alpha]}(x')
    =
    \int_{(Ax_{\ZZ[\alpha]})^-}\int\chi_{\mathcal{K}_\varepsilon}\bd T^{\rm M}_{a_q}(T^{\rm M}_{a_p}(x')) \bd\mu_{\ZZ[\alpha]}(x')
    = 0.
  \end{align}
  Here $\mathcal{K}_\varepsilon$ is defined as in Definition \ref{def: space of measures}.

  \textbf{Analysis on $(Ax_{\ZZ[\alpha]})^{+}$:}
  Let $x'\in (Ax_{\ZZ[\alpha]})^{+}$. By Eqs. \eqref{eq: ax+def} and the definition of $p$ we get that 
  \[\lambda_1(x') \ge p^{-(1-\varepsilon)/n}.\]
  By Theorem \ref{thm: Hecke in the cusp}, we get that $T^{\rm M}_{a_p}(x')$ has no escape of mass as $M \to \infty$.
  Note that here $x'$ may vary as a funciton of $M$.
  By Corollary \ref{cor: equidistribution Of Hecke weak}, we get that 
  \begin{align}\label{eq: equidistribution Taq Tapx'}
    T^{\rm M}_{a_q}(T^{\rm M}_{a_p}(x'))\xrightarrow{M\to \infty}m_{X_n}.
  \end{align}

  \begin{definition}[Spaces of measures and ergodic decomposition]
    Let $X_n^* = X_n\sqcup \{*\}$ denote the one point compactification of $X_n$. This is a compact metrizable space. Hence, by Prohorov's Theorem (See \cite{pinsky1969convergence}), the space of probability measures on $X_n^*$, namely $\mathcal{M}(X_n^*)$, is again a compact metrizable space. Denote the metric on it by $d_{\mathcal{M}(X_n^*)}$. 
    The space of $A$-invariant probability measures, denoted $\mathcal{M}(X_n^*)^A$ is a closed subset, hence again a compact metrizable space. 
    Again, by Prohorov's Theorem, the space of probability measures on $\mathcal{M}(X_n^*)^A$, namely $\mathcal{M}(\mathcal{M}(X_n^*)^A)$, is again a compact metrizable space.
    We interpret $m_{X_n}, \nu_0, T_{B_i}^{\rm M}(\mu_{Ax_{\ZZ[\alpha]}})$ as points in $\mathcal{M}(X_n^*)^A$. 
    Let $\omega = \sum_{i=1}^r \beta_i \delta_{T_{B_i}^{\rm M}(\mu_{Ax_{\ZZ[\alpha]}})} \in \mathcal{M}(\mathcal{M}(X_n^*)^A)$.
    This is the ergodic decomposition of $\nu_0$, and it satisfies $\int \bd\omega = \nu_0$.   
  \end{definition}

  In these terms, we will analyze the limiting behavior we got earlier. 
  Denote 
  \[
    \lim_{M\to \infty}\nu_0 = \vec \nu_0,\qquad \lim_{M\to \infty}\omega = \vec \omega. 
  \]
  Eq. \eqref{eq: - behavior} implies that $\omega$ is supported on measures giving mass at most $c+o_M(1)$ to $\mathcal{K}_\varepsilon$. 
  Taking $M$ to infinity we get that $\vec \omega$ is supported on measures giving mass at most $c$ to $X_n$, that is, giving mass at least $1-c$ to the additional point $*$. In a formula,
  \begin{align}\label{eq: escape of mass in limit}
    \vec \omega\left( \left\{ \mu \in \mathcal{M}(X_n^*):\mu(\{*\}) \ge 1-c \right\} \right) = 1.
  \end{align}

  Hence 
  \begin{align}\label{eq: int nu0 ge}
    (1-c)\delta_{*}\le \int \bd\vec \omega = \lim_{M\to \infty}\int \bd\omega = \vec \nu_0.
  \end{align}

  On the other hand, Eq \eqref{eq: equidistribution Taq Tapx'} implies that the part of $\nu_0$ coming from the application of $T_{a_q}^{\rm M}\circ T_{a_p}^{\rm M}$ to $(Ax_{\ZZ[\alpha]})^{+}$ equidistributes as $M\to \infty$. This implies that $\vec \nu_0 \ge cm_{X_n}$. 
  Together with \eqref{eq: int nu0 ge} we deduce that 
  \begin{align}
    \vec \nu_0 = cm_{X_n} + (1-c)\delta_{*}.
  \end{align}
  This is the ergodic decomposition of $\vec \nu_0$, which implies that $\vec \omega$, whose integral is $\vec \nu_0$, is supported on measures of the form $c'm_{X_n} + (1-c')\delta_{*}$ for some $c'$. By Eq. \eqref{eq: escape of mass in limit} we deduce that almost surely $c' \ge c$.
  Since the integral $\int \bd\omega =cm_{X_n} + (1-c)\delta_{*} $, we deduce that $\vec \omega = \delta_{cm_{X_n} + (1-c)\delta_{*}}$. 
  In particular, %\yinote{This metric here has a different notation than the one of Def. \ref{def: space of measures}.} 
  \[\min_{i=1}^rd_{\mathcal{M}(X^*_n)}(T_{B_i}^{\rm M}(\mu_{Ax_{\ZZ[\alpha]}}), cm_{X_n} + (1-c)\delta_{*}) = \min_{\mu \in \supp(\omega)}d_{\mathcal{M}(X^*_n)}(\mu, \supp(\vec \omega)) \xrightarrow{M\to \infty} 0,\]
  as desired.
\end{proof}
\begin{remark}\label{rem: beta i}
  The constants $\beta_i$ can be computed to be $\frac{\#B_i}{\#B}$. However, this computation requires a disjointness result, namely, $T_{a_q}(y)\cap T_{a_q}(y') = \emptyset$ for all $y\neq y'\in T_p(x_{\ZZ[\alpha]})$. Nonetheless, this is not needed in the proof.
\end{remark}
\begin{remark}[Variants on the proof]\label{rem: one Hecke suffices and polynomial dependence in discriminant}
    \textbf{Using one Hecke operator:}
    In the proof above we use the composition of the Hecke operators $T_{a_q}^{\rm M} \circ T_{a_p}^{\rm M}$. Had we proved the stronger version of Theorem \ref{thm: Hecke in the cusp} as suggested in Remark \ref{rem: can do equidistribution}, we could use $T_{a_p}^{\rm M}$ directly.
    
    \textbf{Obtaining a lower bound on the volume:}
	We sketch a way to perturb the above proof to obtain a lower bound on the volume of the orbits which is polynomial in the discriminant. This lower bound will come from showing that many of the $B_i$'s defined in Eq. \eqref{eq: b-partition} can be made of polynomial size in the discriminant. 
    The $B_i$'s correspond to orbits of action of $\cO_K^{\times}$ on $\PP((\ZZ/pq\ZZ)^n)$. To see that the orbits can be made large, one can use the number field \href{https://en.wikipedia.org/wiki/Artin%27s_conjecture_on_primitive_roots}{Artin's Conjecture},
    which implies that typically one of the $B_i$'s is very large. However, we can sketch a rigorous argument ensuring that most of the $B_i$'s are large. For that, let $r$ be the smallest prime such that $r=1\mod 2n$  and choose the $a_i$'s in Definition \ref{def: shapira's orbit} such that $a_i\mod r$ are the roots of $x^n+1\mod r$. Moreover, choose the $a_i$'s such that $a_1\cdot a_2\cdots a_n\not\equiv(-1)^{n}\mod r^2$.
    This choice of $a_i$ ensures that the corresponding number field $K$ is totally ramified at $r$. This can be used to show that the units in $\cO_K^{\times}$ generate a group of size $\Theta(r^{nm})$ inside $(\cO_K/r^m\cO_K)^{\times}$ for all $m$. Now we replace the Hecke operators $T_{a_{p}}$ in the proof by $T_{a_{r}^m}$. Then, the lower bound on the orbit sizes implies the lower bound on the size of many $B_i$'s and therefore also the bound on the volumes.
\end{remark}
\bibliographystyle{plain}
\bibliography{BibErg}{}
\end{document}